\documentclass[10pt,a4paper]{article}
\usepackage{microtype}

\usepackage[utf8]{inputenc}

\usepackage{a4wide}
\usepackage{fullpage}
\usepackage{latexsym}
\usepackage{amssymb,amsmath,amsthm}
\usepackage{graphicx}
\usepackage{url}
\usepackage{hyperref}
\usepackage{xcolor}
\usepackage{enumerate}
\usepackage[capitalize]{cleveref}
\usepackage{multirow}
\usepackage{authblk}
\usepackage[normalem]{ulem}
\usepackage{mathtools} 

\newtheorem{theorem}{Theorem}
\newtheorem{proposition}[theorem]{Proposition}

\newtheorem{lemma}[theorem]{Lemma}
\newtheorem{corollary}[theorem]{Corollary}
\newtheorem{question}[theorem]{Question}
\newtheorem{problem}[theorem]{Problem}
\newtheorem{conj}[theorem]{Conjecture}

\newtheorem*{claim*}{Claim}
\theoremstyle{definition}

\theoremstyle{remark}
\newtheorem{remark}[theorem]{Remark}
\newtheorem{remarks}[theorem]{Remarks}

\makeatletter
\newcommand{\ProofEndBox}{{\ifhmode\unskip\nobreak\hfil\penalty50 \else
          \leavevmode\fi\quad\vadjust{}\nobreak\hfill$\Box$
            \finalhyphendemerits=0 \par}}
\makeatother

\newcommand{\R}{{\mathbb{R}}}
\newcommand{\C}{{\mathbb{C}}}

\DeclareMathOperator{\im}{im}

\DeclareMathOperator{\alts}{alt-sym}
\DeclareMathOperator{\sym}{sym}

\newcommand{\emb}{\textsc{Embed}}

\newcommand{\Z}{{\mathbb{Z}}}

\newcommand{\sgn}{\mathop {\rm sgn}\nolimits}

\newcommand\FF{\mathcal{F}}
\newcommand\GG{\mathcal{G}}

\newcommand\HH{\mathbb{H}}

\renewcommand\AA{\mathbf{A}}
\newcommand\BB{\mathbf{B}}
\newcommand\II{\mathbf{I}}
\newcommand\QQ{\mathbf{Q}}
\newcommand\XX{\mathbf{X}}

\newcommand\makevec[1]{{\bf #1}}

\def \oo {\makevec{o}}

\def \yy {\makevec{y}}

\newcounter{sideremark}
\newcommand{\marrow}{\stepcounter{sideremark}\marginpar{$\boldsymbol{\longleftarrow\scriptstyle\arabic{sideremark}}$}}

\newif\ifcmts
\cmtstrue

\ifcmts
\newcommand{\pavel}[1]{{\color{teal}\vskip 5pt\textsf{*** (Pavel) \marrow #1\vskip 5pt}}}
\newcommand{\martin}[1]{{\color{blue}\vskip 5pt\textsf{*** (Martin) \marrow #1\vskip 5pt}}}
\else
\newcommand{\pavel}[1]{}
\newcommand{\martin}[1]{}
\fi

\title{Embeddings of $k$-complexes into $2k$-manifolds.\thanks{This research is
    supported by the GA\v{C}R grant 19-04113Y.}}

\author[1,2]{Pavel Pat\'{a}k}
\author[1]{Martin Tancer}

\affil[1]{\small Department of Applied Mathematics, Charles University, Malostransk\'{e} n\'{a}m.
25, 118~00~~Praha~1, Czech Republic}
\affil[2]{Czech Technical University in Prague, Faculty of Information Technology, Th\'{a}kurova 2700/9, 160~00~~Praha~6, Czech Republic}

\begin{document}

\maketitle
\begin{abstract}
 We improve the bound on Kühnel's problem to determine the smallest $n$
  such that the $k$-skeleton
  of an $n$-simplex $\Delta_n^{(k)}$ does not embed into a compact PL $2k$-manifold $M$ by showing 
  that if $\Delta_n^{(k)}$ embeds into $M$, then $n\leq (2k+1)+(k+1)\beta_k(M;\mathbb Z_2)$.
  As a consequence we obtain improved Radon and Helly type results for set systems in such manifolds.

  Our main tool is a new description of an obstruction for embeddability of a
  $k$-complex $K$ into a compact PL $2k$-manifold $M$ via the intersection form on $M$. In our approach we need that for every map $f\colon K\to M$
the restriction to the $(k-1)$-skeleton of $K$ is nullhomotopic. In particular, 
this condition is satisfied in interesting cases if $K$ is $(k-1)$-connected,
  for example a $k$-skeleton of $n$-simplex, or if $M$ is $(k-1)$-connected.
  In addition, if $M$ is $(k-1)$-connected and $k\geq 3$, the
  obstruction is complete, meaning that a $k$-complex
$K$ embeds into $M$ if and only if the obstruction vanishes.
For trivial intersection forms, our obstruction coincides with the standard van
  Kampen obstruction. However, if the form is non-trivial, the obstruction is not linear but rather 'quadratic' in a sense
that it vanishes if and only if certain system of quadratic diophantine
  equations is solvable. This may potentially be useful in attacking
  algorithmic decidability of embeddability of $k$-complexes into PL
  $2k$-manifolds.
\end{abstract}

\section{Introduction}

\subparagraph{Motivation.}
This paper has three main goals:
\begin{enumerate}[1)]
 \item Describe an obstruction for (almost)-embeddability of $k$-dimensional
   simplicial complexes into compact $2k$-dimensional PL manifolds
    (\Cref{t:obstruction,t:complete}). This extends the standard van Kampen obstruction for embeddability into $\R^{2k}$.
 \item Improve the bounds for so-called K\"u{hnel} problem: Provide an upper
   bound on $n$ such that the $k$-skeleton of the $n$-simplex embeds into
    compact $2k$-dimensional PL-manifold $M$ (\Cref{t:kuhnel_improved}).
 \item Use the bounds from the previous item to obtain versions of Radon's and Helly's theorem on manifolds (\Cref{thm:radon,cor:helly}).
\end{enumerate}

Motivation for such research emerges from various directions:

The classical setting, related to the first goal, considers the case of embeddings of $k$-complexes into $\R^{2k}$ (a $k$-complex is always embeddable into $\R^{2k+1}$, thus the target dimension $2k$ is the first nontrivial dimension).
This line of research was initiated by results of van Kampen and
Flores~\cite{vankampen32,flores32} on nonembeddability of the $k$-skeleton of
the $(2k+2)$-simplex, $\Delta_{2k+2}^{(k)}$, and the $(k+1)$-fold join of three
isolated points into $\R^{2k}$. This case is in general well understood:
If $k \neq 2$, embeddability of $K$ in $\R^{2k}$ is characterized via vanishing of so-called
van Kampen
obstruction~\cite{vankampen32,shapiro57,wu65,melikhov09},
which is even efficiently computable (details on computability are given
in~\cite{matousek-tancer-wagner11}). If $k=2$, the obstruction is
incomplete~\cite{freedman-krushkal-teichner94}, and it seems to be a
challenging problem to determine whether embeddability of $2$-complexes into
$\R^4$ is decidable. (Only NP-hardness is
known~\cite{matousek-tancer-wagner11}.)
However, there are many interesting target spaces that are not $\R^{2k}$.
In geometry one often works with projective spaces, incidence problems lead to
embeddings into Grassmanians or flag
manifolds, etc. A possible concrete example where the ideas of this
paper can be useful are considerations of Helly type results as
in~\cite{goaoc-patak-patakova-tancer-wagner17}.\footnote{However, our work
  should be understand only as a first step towards an improvement
of~\cite{goaoc-patak-patakova-tancer-wagner17}. In particular, we did not
attempt to upgrade our results to \emph{homological almost embeddings} which
are really used in~\cite{goaoc-patak-patakova-tancer-wagner17}.} Here considerations of a
general manifold $M$ become apparent, for example, when considering Helly-type
theorems for line transversals as
in~\cite{cheong-goaoc-holmsen-petitjean08}.

In relation to our second main goal, embeddability of the $k$-skeleton of $n$-simplex $\Delta^{(k)}_n$ to a $2k$-manifold 
was
considered in \cite{kuhnel94, volovikov96,
goaoc-mabillard-patak-patakova-tancer-wagner17}. Volovikov~\cite{volovikov96} shows, for quite 
general $M$, that there is no embedding $f\colon |\Delta^{(k)}_{2k+2}| \to M$
provided that $f$ induces a trivial map on (co)homology, which generalizes
nonembeddability of $\Delta_{2k+2}^{(k)}$ in $\R^{2k}$.\footnote{Volovikov's
result is in fact even more general in different directions.} Given a
$(k-1)$-connected $2k$-manifold $M$ such that $k$-skeleton of $n$-simplex embeds
into $M$, K\"{u}hnel conjectured an upper bound on $n$ depending only on $k$
and the Euler characteristic of $M$; see equation~\eqref{e:Kuhnel_chi}
below. A weaker bound was proved in
\cite{goaoc-mabillard-patak-patakova-tancer-wagner17}. As an application of our
tools, we will show how this bound can be significantly improved (for
compact PL manifolds\footnote{In fact, we are not sure from the statement of
the conjecture in~\cite[Conjecture~B]{kuhnel94} whether it regards arbitrary
(possibly non-compact) manifolds or whether it regards polyhedral manifolds
discussed in the paper which are PL and closed (a fortiori compact).}). Such improvement also yields improved Radon and Helly type theorems on manifolds, which is our third goal.

Finally, in a special case when $k=1$ our research coincides with a classical topic of embeddings
of graphs in surfaces~\cite{mohar-thomassen01}; and, in particular, our work is
related to Hanani--Tutte type results for graphs on
surfaces~\cite{pelsmajer-schaefer-stasi09, fulek-kyncl19,
fulek-kyncl18socg}. In the language of these references, our algebraic
description in this case provides a characterization of graphs admitting an
independently even drawing into a given surface.

\subsection{The K\"{u}hnel problem and Helly-type results}
Before we explain the details of our description of embeddability of
$k$-complexes into $2k$-manifolds, let us survey a few results that we can
reach with our tools.

\subparagraph{K\"{u}hnel's conjecture.}
Kühnel conjectured~\cite{kuhnel94} that if the $k$-dimensional skeleton $K:=\Delta_n^{(k)}$ can be embedded into a $(k-1)$-connected $2k$-manifold $M$,
then \begin{equation}
\label{e:Kuhnel_chi}
\binom{n-k-1}{k+1}\leq (-1)^k\binom{2k+1}{k+1}(\chi(M) - 2).
     \end{equation}
Because of $(k-1)$-connectivity, this inequality is equivalent to 
\begin{equation}
  \label{e:kuhnel}
  \binom{n-k-1}{k+1}\leq \binom{2k+1}{k+1}\beta_k(M;\Z_2),
\end{equation}
which seems to hold even without the connectivity assumption. (Here
$\beta_k$ denotes the Betti number.)

The special case $k=1$, of Kühnel's conjecture is known as Heawood
inequality and it is fully confirmed in this case (see~\cite{ringel74} for
discussion).\footnote{Note that $\Delta_n^{(k)}$ has $n+1$ vertices, thus $n$
is shifted by one when compared with the standard statement of the Heawood
inequality.} 
For $k \geq 2$, a recent far reaching work Adiprasito~\cite{adiprasito18_arxiv} proves the K\"{u}hnel
bound under an additional assumption that the embedding is sufficiently tame.
Without the tameness assumption, together with Goaoc, Mabillard, Pat\'{a}kov\'{a} and
Wagner~\cite{goaoc-mabillard-patak-patakova-tancer-wagner17}, we have obtained
a bound $n \leq 2\beta_k(M;\Z_2)\binom{2k+2}k + 2k + 4$. Here we demonstrate how the
`obstruction machinery' may improve this bound (under an extra assumption that
$M$ is PL). Once the machinery is set up, the main idea of the proof is
relatively simple; see the sketch at the beginning of
Section~\ref{s:kuhnel}.

Given a simplicial complex $K$ and a manifold $M$ (or arbitrary topological
space in general) an \emph{almost embedding} of $K$ into $M$ is a map $f \colon |K| \to M$ such
that $f(\sigma) \cap f(\tau) = \emptyset$ whenever $\sigma$ and $\tau$ are
disjoint simplices of $K$.
Every embedding is an almost embedding.
By $\Omega_{\Z_2} \colon H_k(M;\Z_2) \times H_k(M;\Z_2) \to \Z_2$ we will denote the
$\Z_2$-intersection form on $M$. 
The intersection form is discussed in more detail in
Subsection~\ref{ss:if}.
The main properties are that $\Omega_{\Z_2}$
is a symmetric bilinear form and if $z$ and $z'$ are two general position
$k$-cycles in $M$, then $\Omega_{\Z_2}([z],[z'])$ counts the number of
crossings between $z$ and $z'$ modulo $2$; here $[\cdot]$ stands for the
corresponding homology class.

\begin{theorem}
\label{t:kuhnel_improved}
 If the $k$-skeleton $\Delta_n^{(k)}$ of an $n$-simplex can be almost embedded
  into a compact (possibly with boundary) PL $2k$-manifold $M$,
 then 
 \begin{enumerate}[$(i)$]
   \item
   $n\leq (2k+1) + (k+1)\beta_k(M;\Z_2)$ and
 \item
   $n\leq (2k+1) + \frac12(k+2)\beta_k(M;\Z_2)$ if the intersection form
     on $M$ is alternating, that is $\Omega_{\Z_2}(h,h)=0$
    for all $h\in H_k(M;\Z_2)$.
  \end{enumerate}
\end{theorem}

If $\beta_k(M;\Z_2)=1$, our bounds agree with the value proposed by Kühnel
and if the form is alternating
the same is true for
$\beta_k(M;\Z_2)=2$. The condition that the form is alternating is a natural
condition that occurs, for example, if $M$ is a connected sum of $S^k \times
S^k$. One of the advantages of Theorem~\ref{t:kuhnel_improved}
is that it also applies to manifolds which are not $(k-1)$-connected. This
distinguishes it from Kühnel's conjecture. Using
Theorem~\ref{t:kuhnel_improved} we can, for example, see
that there is no (almost) embedding of $\Delta_{12}^{(3)}$ into $\R P^6$.

For $k=1$, Theorem~\ref{t:kuhnel_improved} does not recover the Heawood
inequality. However, considering that Theorem~\ref{t:kuhnel_improved} is stated
for almost embeddings, it seems to say something new even for $k=1$ as it is an
open question whether embeddability and almost embeddability coincide for
graphs on surfaces~\cite[Problem~5.2]{fulek-kyncl19}.
  Almost embeddability is relevant for example in context of Helly-type
theorems; see Theorem~\ref{thm:radon} and Corollary~\ref{cor:helly} below.

There are several cases where the inequalities from Theorem~\ref{t:kuhnel_improved}
are tight:
there is a $6$-point triangulation of the real projective plane ($k=1$, $\beta_1(M;\Z_2) = 1$, $n=5$),
a $9$-point triangulation of the complex projective plane ($k=2$, $\beta_2(M;\Z_2) = 1$, $n=8$)~\cite{Kuhnel1983} and 15-point triangulation of the quaternionic projective plane ($k=4$, $\beta_4(M;\Z_2)=1$, $n=14$)~\cite{Gorodkov19},
and the torus can by triangulated using $7$ vertices only ($k=1$, $\beta_1(M;\Z_2) = 2$, $n=6$). Quick computation of the number of faces reveals that each of these triangulations necessarily contains the complete $k$-skeleton of $\Delta_n$. (It is $(k+1)$-neighbourly~\cite{kuhnel94}.)

In addition, there is a hope that bounds of
Theorem~\ref{t:kuhnel_improved} can be still improved significantly by using
our tools, possibly giving a solution of the K\"{u}hnel conjecture. 
In Section~\ref{s:kuhnel} we pose a specific conjecture (purely in
combinatorics and linear algebra), Conjecture~\ref{c:kuhnel_lambda}, that
implies K\"{u}hnel's conjecture (in case that $M$ is a compact PL-manifold). A computer
assisted search for small values of $k$ and $\beta_k(M;\Z_2)$ suggests that
Conjecture~\ref{c:kuhnel_lambda} may hold.

\paragraph{Radon and Helly type theorems.} 
Improved bounds on the K\"{u}hnel problem as in
Theorem~\ref{t:kuhnel_improved} immediately imply improved bounds on the Radon number (value $r$ in the statement below) in the theorem below. Consequently one obtains better bounds on Helly's number~\cite{Levi1951}, Tverberg's numbers~\cite{jamison1981}, fractional Helly number~\cite{boundedRadon_fractHelly}, existence of weak $\varepsilon$-nets and $(p,q)$-theorems~\cite{boundedRadon_fractHelly, alon2002}.

\begin{theorem}\label{thm:radon}
  Let $M$ be a compact PL $2k$-manifold. 
 Let $\operatorname{cl}\colon 2^M\to 2^M$ be a closure operator.\footnote{A closure operator is any function $\operatorname{cl}\colon 2^M\to 2^M$ that for all $S,R\subseteq M$ satisfies $S\subseteq \operatorname{cl} S$, $R\subseteq S\Rightarrow \operatorname{cl}(R)\subseteq\operatorname{cl}(S)$ and $\operatorname{cl}\left(\operatorname{cl} S\right)=\operatorname{cl}S$. Typical examples are convex and affine hulls in $\R^d$ or topological closure operators in topological spaces.}
 Let $P\subseteq M$ be a set of size $r$,
 such that $\operatorname{cl}S$ is (topologically) $k$-connected\footnote{From the proof follows that it suffices to require that the $(|S|-1)$-th homotopy group of  $\operatorname{cl}(S)$ is trivial for each $S$ of size at most $k+1$. Moreover, if one is willing to increase the bound on $r$, it is possible to use the ideas from~\cite{matousek97} or~\cite{patak19} and allow that $\operatorname{cl}(S)$ have more path-connected components, if all of them are sufficiently connected.} for every $S\subseteq P$ of size at most $k+1$.
 
 \begin{enumerate}[(i)]
   \item If $r \geq 2k+3 + (k+1)\beta_k(M;\Z_2)$, or
   \item if the intersection form of $M$ is alternating (over $\Z_2$) and $r \geq 2k+3 +
   \frac12(k+2)\beta_k(M;\Z_2)$, 
  \end{enumerate}
     then there are two disjoint subsets $P_1,P_2\subseteq P$
  such that $\operatorname{cl}(P_1)\cap\operatorname{cl}(P_2)\neq\emptyset$. 
\end{theorem}

\begin{corollary}[Helly-type theorem]\label{cor:helly}
  Let $M$ be a compact PL $2k$-manifold. Let $\FF$ be a finite collection of subsets of $M$ such that $\bigcap \GG$ is $k$-connected or empty for every subfamily $\GG \subsetneq \FF$.
  If $\bigcap \GG$ is nonempty for every $\GG\subseteq \FF$ of cardinality less than $r$, where $r$ is as in the previous theorem, then $\bigcap \FF\neq\emptyset$.
\end{corollary}

\begin{proof}[Proof~\cite{Levi1951,Radon1921}]
 Consider the following closure operator \[\operatorname{cl}_\FF(S):=\bigcap_{\substack{F\in\FF \\ S\subseteq F}} F.\]
  If $\bigcap \FF=\emptyset$, let $\GG\subseteq \FF$ be a minimal subset with $\bigcap \GG=\emptyset$. By assumption $|\GG|\geq r$.
  By minimality of $\GG$ for each $G\in \GG$ there is a point $y_G\in\bigcap\GG\setminus \{G\}$. 
 The set $\{y_G\mid G\in\GG\}$ has at least $r$ points. Thus Theorem~\ref{thm:radon} guarantees that it can be split into two disjoint sets $P_1\sqcup P_2$ such that there is a point $y\in\operatorname{cl}_\FF(P_1)\cap \operatorname{cl}_\FF(P_2)$. By our choice of the closure operator, such $y$ lies in every set of $\GG$, contradicting $\bigcap\GG\neq\emptyset$.
\end{proof}

Theorem~\ref{thm:radon} follows the line of research of deducing Helly and Radon type theorems
from non-embeddability
results; see \cite{matousek97}, \cite{goaoc-patak-patakova-tancer-wagner17} or \cite{patakova2019}.

Corollary~\ref{cor:helly} is an analogy of Theorem 2 in~\cite{matousek97} or Theorem~1
in~\cite{goaoc-patak-patakova-tancer-wagner17} for manifolds, with stronger
assumption on intersections. This is already interesting for manifolds, as the 
optimal Helly number is linked to the solution of the K\"{u}hnel problem.
The proof of Theorem~\ref{thm:radon} (modulo
Theorem~\ref{t:kuhnel_improved}) follows by a combination of a suitable definition of Radon number~\cite{patakova2019} and techniques developed in \cite{matousek97,
goaoc-patak-patakova-tancer-wagner17}. 
Because the proof is short, we reproduce
it immediately.

\begin{proof}[Proof of Theorem~\ref{thm:radon}]
For contradiction, we assume that 
$\operatorname{cl} P_1\cap \operatorname{cl} P_2$ is empty for every two disjoint subsets $P_1, P_2$ of $P$. 
Under this assumption,
  we will build an almost embedding $f\colon \Delta^{(k)}_{r-1} \to M$. This
  contradicts Theorem~\ref{t:kuhnel_improved}(i) in case (i) and
  Theorem~\ref{t:kuhnel_improved}(ii) in case (ii).    
  We define $f$ inductively, skeleton by skeleton.
  We start by letting the points $P$ be the $0$-skeleton of $\Delta_{r-1}$.
During the construction, we maintain the following property: If $\sigma$ is a
  simplex of dimension at most $k$ and $I$ is the set of vertices of $\sigma$,
  then $f$ maps $\sigma$ into $\operatorname{cl}(I)$.

  Now, given a simplex $\sigma$ of dimension at most $k$, assume that $f$ is  already defined on $\partial
  \sigma$. Due to the property we maintain, we get that $f(\partial \sigma)$
  belongs to $\operatorname{cl}I$, where $I$ is the set of vertices of $\sigma$.
  As $\operatorname{cl}I$ is $k$-connected, we can extend $f$ to $\sigma$ inside
  $\operatorname{cl}I$, thus we maintain the required property.
  It remains to show that the resulting $f$ is an almost embedding of
  $\Delta_{r-1}^k$ into $M$. Given disjoint $k$-simplices $\sigma$ and $\tau$
  of $\Delta_{r-1}$, let $I$ be set of vertices of $\sigma$ and $J$ be the set
  of vertices of $\tau$. In particular $I$ and $J$ are disjoint. But then $f(\tau)$ lies in $\operatorname{cl} I$ and $f(\sigma)$ lies in $\operatorname{cl} J$ and these two sets are disjoint by our assumption.
\end{proof}

\subsection{Obstruction for embeddability}
\label{ss:obstruction}
Now we describe an obstruction for embeddability of a $k$-complex into a
compact PL $2k$-manifold,
which is our main technical tool. 
In general, we follow~\cite{shapiro57,freedman-krushkal-teichner94, johnson02, skopenkov08,
melikhov09} and \cite[App. D]{matousek-tancer-wagner11}; however the
concrete interpretations of the van Kampen obstruction in these references
somewhat vary.
We choose to specify the details in a way convenient for working
with intersection form later on. We postpone the precise definition of a
general position map, intersection number and intersection form to
Section~\ref{s:prelim} as they are not so essential for understanding this text
in the introduction.

\subparagraph{The standard van Kampen obstruction.} 
Let $k \geq 1$ and $K$ be a simplicial $k$-complex.
Let $f\colon |K| \to
\R^{2k}$ be a general position map. Given two disjoint $k$-simplices $\sigma$
and $\tau$ of $K$, the number of intersections $f(\sigma)$ and $f(\tau)$ is
finite and each such intersection is transversal. One way how to express the
idea of the van Kampen obstruction~\cite{vankampen32} is the following: Let
$\Z_2^P$ be the vector space over $\Z_2$ whose coordinates are indexed by the
set $P$ of all (ordered) pairs $(\sigma, \tau)$ of disjoint $k$-simplices of
$K$. The general position map $f$ induces a vector $v_f \in \Z_2^P$ such that
its coordinate corresponding to the pair $(\sigma, \tau)$ is the number of
intersections between $f(\sigma)$ and $f(\tau)$ modulo 2. It turns out that the
vectors $v_f$ when considering over all possible general maps $f$ form an
affine subspace $A$ of $\Z_2^P$. In particular, if there is an embedding $g$ of $K$
into $\R^{2k}$, this affine subspace $A$ has to contain the zero vector $v_g$.
For concrete $K$, it is possible to determine whether $A$ contains the
zero vector, and $A$ is essentially the object that we will call the van Kampen
obstruction (see below).

For practical purposes (computations), it is convenient to consider $A$ as a
certain cohomology class which we will overview below. In particular, $P$ will
be replaced with deleted product of $K$; $v_f$ with corresponding intersection
cochain and $A$ with certain cohomology class denoted $\oo(K)$. In addition,
the similar ideas as above may be performed over the integers $\Z$ instead of
$\Z_2$. The cost is that one has to consider intersections of $f(\sigma)$ and
$f(\tau)$ carefully with signs but the benefit is that the integer valued 
obstruction is complete for $k \neq 2$.

From now on we perform all our considerations in a ring $R = \Z_2$ or $R =
\Z$. All the orientation considerations can be skipped if $R = \Z_2$. (This specifically applies in the proof of Theorem~\ref{t:kuhnel_improved} as the $\Z_2$-version of the obstruction is fully sufficient there.)
Let
$\tilde K := \{\sigma \times \tau\colon \sigma, \tau \in K, \sigma \cap \tau = \emptyset\}$
denote the \emph{deleted product} of $K$. We fix an orientation of every
simplex of $K$. This induces an orientation of the cells $\sigma \times \tau
\in \tilde K$ by the product orientation.\footnote{If $(u_1, \dots, u_p)$ is a
positive basis of $\sigma$ and $(v_1, \dots, v_q)$ is a positive basis of
$\tau$, then $((u_1,0), \dots, (u_p,0), (0,v_1), \dots, (0,v_q))$ is a positive basis of
$\sigma \times \tau$.} By $C_m(\tilde K; R)$ we denote the group of
$m$-chains in $\tilde K$ (for some integer $m$).\footnote{We work with cellular homology, thus 
the group $C_{m}(\tilde K; R)$ should be understood as
$H_{m}(\tilde K^{(m)}, \tilde K^{(m-1)}, R)$ and $\sigma \times \tau$ should be
understood as an oriented generator corresponding to the cell $\sigma \times
\tau$ with $\dim \sigma + \dim \tau = m$. The symbol $\tilde K^{(i)}$
stands for $i$-skeleton of $\tilde K$.} 
This essentially means that $C_m(\tilde K; R)$ is the group of
formal $R$-combinations of products $\sigma \times \tau$ with the
fixed orientation as above. The boundary operator on $C_{m}(\tilde K;
R)$ is given by
\begin{equation}
  \partial (\sigma \times \tau) = (\partial \sigma) \times \tau + (-1)^{\dim
  \sigma} \sigma \times (\partial \tau).
\end{equation}
By $C^{m}(\tilde K; R)$ we denote the group of $m$-cochains in $\tilde K$.
These are homomorphisms from $C_{m}(\tilde K; R)$ to $R$. The coboundary
operator, dual to the boundary operator, is given by
\begin{equation}
  \delta \xi (\sigma \times \tau) = \xi((\partial \sigma) \times \tau) + (-1)^{\dim
  \sigma} \xi(\sigma \times (\partial \tau))
\end{equation}
for $\xi \in C^{m}(\tilde K; R)$.
We will need (only in dimension $2k$) a subgroup $C^{2k}_{\alts}(\tilde K; R)$ of $C^{2k}(\tilde
K; R)$ consisting of cochains $\xi \in C^{2k}(\tilde
K; R)$ satisfying
$$
\xi(\sigma \times \tau) = (-1)^k\xi(\tau \times \sigma).
$$
We will also need (only in dimension $2k-1$) a subgroup $C^{2k-1}_{\sym}(\tilde
K; R)$ of $C^{2k-1}(\tilde
K; R)$ consisting of cochains $\xi \in C^{2k-1}(\tilde
K; R)$ satisfying
$$
\xi(\sigma \times \tau) = \xi(\tau \times \sigma).
$$
We call $\xi \in C^{2k}_{\alts}(\tilde K; R)$ \emph{alternately-symmetric}
and $\xi \in C^{2k-1}_{\sym}(\tilde
K; R)$ \emph{symmetric}. A simple computation reveals that $\partial \xi \in
C^{2k-1}_{\sym}(\tilde
K; R)$ for $\xi \in C^{2k}_{\alts}(\tilde K; R)$ and vice versa $\delta \xi \in
C^{2k}_{\alts}(\tilde K; R)$ for $\xi \in C^{2k-1}_{\sym}(\tilde
K; R)$. Then the cohomology group $H^{2k}_{\alts}(\tilde K; R)$ is defined in
the standard way as $H^{2k}_{\alts}(\tilde K; R) = \ker \delta_{2k} / \im
\delta_{2k-1}$ with respect to the coboundary operator 
\begin{equation}
  \label{e:coboundary_operators}
  C^{2k-1}_{\sym}(\tilde K; R) \xrightarrow{\delta_{2k-1}}
C^{2k}_{\alts}(\tilde K; R)\xrightarrow{\delta_{2k}} 0.
\end{equation}
The operator $\delta_{2k}$ is in particular
trivial, thus $\ker \delta_{2k} = C^{2k}_{\alts}(\tilde K; R)$.

Given a general position map $f\colon |K| \to \R^{2k}$, we have the
\emph{intersection cochain} $\vartheta_f \in C_{2k}(\tilde K; R)$ given so
that $\vartheta_f(\sigma \times \tau)$ is the intersection number of
$f(\sigma)$
and $f(\tau)$. (The details are postponed to Section~\ref{s:prelim}.
Intuitively, the intersection number is the number of intersections between
$f(\sigma)$ and $f(\tau)$; however, if $R = \Z$, then the intersections have to
be counted carefully with signs.) This cochain satisfies 
$\vartheta_f(\sigma \times \tau) = (-1)^k \vartheta_f(\tau \times \sigma)$;
therefore it belongs to $C^{2k}_{\alts}(\tilde K; R)$. It turns out that the
cohomology class $[\vartheta_f] \in H^{2k}_{\alts}(\tilde K; R)$ is independent
of the choice of $f$. This class (for arbitrary $f$) is called the \emph{van Kampen obstruction}
for embeddability of $K$ into $\R^{2k}$ and we will denote it $\oo(K)$. If $f$
is an embedding, then $\vartheta_f = 0$ which also implies that $\oo(K) =
[\vartheta_f] = 0$.  Thus $\oo(K)$ is indeed an obstruction for embeddability
of $K$ into $\R^{2k}$.

\paragraph{The obstruction in a manifold.}
Now let us in addition assume that $M$ is a compact PL $2k$-manifold. Let
us also assume that $M$ is $R$-orientable---this condition is vacuous if $R =
\Z_2$ while this is the standard orientability if $R = \Z$.
By $\Omega\colon H_k(M;R) \times H_k(M;R) \to R$ we denote
the $R$-intersection form on $M$.\footnote{It would be more appropriate
to use the notation $\Omega_R$ instead of $\Omega$ consistently with the
previous subsection. However, from now on, we want to simplify the notation for
$\Omega$.}
If $R = \Z_2$ the properties of the form were sketched in the previous
subsection and they are analogous for $R=\Z$. In general, $\Omega$ is
again alternately-symmetric, that is, $\Omega(h,h') = (-1)^k\Omega(h',h)$; this is of course the same as symmetric if $R = \Z_2$.
Given a homomorphism $\psi \colon C_k(K; R) \to H_k(M; R)$, we define $\omega_\psi
\in C^{2k}_{\alts}(\tilde K)$ by $\omega_\psi(\sigma \times \tau) :=
\Omega(\psi(\sigma), \psi(\tau))$. By alternating symmetry of $\Omega$ we get that
$\omega_\psi$ is indeed an alternately symmetric cochain.

\begin{theorem}[Existence of the obstruction\footnote{As our obstruction is parametrized by homomorphisms
  $\psi\colon C_k(K;R) \to H_k(M; R)$ we have been asked whether
  Theorem~\ref{t:obstruction} can be equivalently stated in (co)homological
  invariants instead of (co)chains. This is indeed possible:
    A homomorphism $\psi\colon C_k(K; R) \to H_k(M;R)$ is an element of the
      cochain group $C^k(K; H_k(M;R))$. Each such element is a cocyle because
      $K$ is $k$-dimensional. It can be computed that $[\omega_\psi]$ is independent
      of the choice of representative $\psi$ of a cohomology class in $H^k(K;
      H_k(M;R))$; thus $[\omega_\psi]$ could be defined only with respect to
      such a cohomology class. 
  }]
\label{t:obstruction}

  Let $k\geq 1$, $R = \Z$ or $R = \Z_2$, $K$ be a $k$-complex, and $M$ be a
  compact $R$-orientable PL $2k$-manifold.
  Assume that
  there is an almost embedding $f \colon |K| \to M$. Assume also that
  the restriction of $f$ to the $(k-1)$-skeleton $K^{(k-1)}$ is nullhomotopic.
Then there is a homomorphism $\psi\colon C_k(K; R) \to H_k(M;R)$ such that
\[ [\omega_\psi] - \oo(K) = 0.\]
\end{theorem}

First, let us remark that the extra assumption that the restriction of $f$ to
the $(k-1)$-skeleton $K^{(k-1)}$ is nullhomotopic is always satisfied in two
important cases: if either $M$ or $K$ is $(k-1)$-connected. In
particular, this occurs if $K := \Delta_n^{(k)}$ is the $k$-skeleton of an
$n$-simplex. The latter one we use in the proof of
Theorem~\ref{t:kuhnel_improved} and a reader interested only in the proof of
Theorem~\ref{t:kuhnel_improved} and willing to accept
Theorem~\ref{t:obstruction} as a blackbox may immediately jump to
Section~\ref{s:kuhnel}.

With slight abuse of terminology, we can consider non-existence of a homomorphism
$\psi$ from the theorem as an \emph{obstruction for (almost) embeddability of $K$ to
$M$}, and we say that this obstruction \emph{vanishes} if such homomorphism
exists. 

\begin{remarks}
  \label{r:obstruction}

  \begin{enumerate}[(a)]
  \item If $\Omega$ is trivial, then $\psi$ must be a trivial homomorphism,
    thus our obstruction coincides with the standard van Kampen obstruction.
  \item The minus sign at $\oo(K)$ in the statement is not important as the van
    Kampen obstruction is an element of order $2$, $\oo(K)  = - \oo(K)$.
  \item We will show that our obstruction is `quadratic' in a sense that it
    vanishes if and only if certain system of quadratic equations has a
    solution; see Theorem~\ref{t:quadratic}. 
  \item Given a map $f\colon |K| \to M$, there are several ways how to
    describe an obstruction, depending on $f$, for existence of a homotopy from
      $f$ to an embedding:
      \begin{enumerate}[1.]
	\item
	  A necessary condition for existence of such homotopy
      is existence of an equivariant homotopy from $f^2 \colon |K|^2 \to M^2$
      to so called isovariant map; see Harris~\cite{harris69} for details. If
      $K$ is a $k$-complex and $M$ is an $m$-manifold and $3k \leq 2m-3$ (in
      particular if $m = 2k$ and $k \geq 3$), then
      this is even `if and only if' condition; see~\cite[Theorem~1]{harris69}.
      This gives rise to obstruction theories in this setting; see Corollary~6 and
      Corollary~8 in~\cite{harris69}. From this point of view, some description
      of an obstruction for embedding $k$-complexes into $2k$-manifolds is not
      new. However, the added value of Theorem~\ref{t:obstruction} is that it
      provides quite concrete description for \emph{all} maps $f\colon |K| \to
      M$ suitable for applications.

    \item A more explicit description of such obstruction appears in a work of Johnson~\cite{johnson02}
	  (in the setting when $K$ is a $k$-complex and $M$ is a
	  $2k$-manifold).
      There are some mild differences in the assumptions on $M$. In
      particular, Johnson works in the smooth case. However, Johnson's
      setting is overall closer to our setting than Harris' setting because he
      essentially works with the van Kampen obstruction. When adapted to
      our notation, Johnson's obstruction is a class in $H^{2k}_{\alts}(\tilde
      K; \Z)$. However, it does not seem that Johnson's approach answers which
	  class is it. We in principle provide this answer (see the proof
	  of Proposition~\ref{p:transfer}) as an intermediate step in a proof of
	  Theorem~\ref{t:obstruction}, though we need to assume the nullhomotopy
	  condition as in Theorem~\ref{t:obstruction}.
      \end{enumerate}
  \end{enumerate} \end{remarks}

As a counterpart to Theorem~\ref{t:obstruction}, using the standard tools, we will show that our obstruction is complete, if $k \geq 3$ and $M$ is
$(k-1)$-connected. (We will mainly follow~\cite{freedman-krushkal-teichner94}
but similar ideas go back at least to Whitney~\cite{whitney44}, Shapiro~\cite{shapiro57} and
Wu~\cite{wu65}.)

\begin{theorem}[Completeness of the obstruction]
\label{t:complete}
  Let $k \geq 3$, $K$ be a $k$-complex, $M$ be a compact $(k-1)$-connected
  (in particular orientable) PL
  $2k$-manifold. Assume that there is a homomorphism $\psi\colon C_k(K; \Z) \to
  H_k(M;\Z)$ such that $[\omega_\psi] - \oo(K) = 0$ (over the integers), 
  that is, the obstruction vanishes. Then there is a PL embedding $f \colon |K| \to M$.
\end{theorem}

\paragraph{The K\"{u}hnel problem revisited.}
Theorem~\ref{t:complete} can be used to transfer the solution to Kühnel's
problem from one manifold to another, as we discuss now.
Another case, where Theorem~\ref{t:complete} could be useful are the
computational aspects that we will discuss in next subsection. In particular,
in both these cases, the consideration of the obstruction over the integers is unavoidable.

In the proof of Theorem~\ref{thm:radon}, it was crucial that Theorem~\ref{t:kuhnel_improved}
holds for almost embeddings (and not only for embeddings). However, our
approach allows, under mild conditions on the manifold, to extend an
upper bound on the K\"{u}hnel problem from embeddings to almost embeddings.
This would be in particular interesting, if it were possible to remove the
additional assumption on the embeddings in Adiprasito's proof of the K\"{u}hnel
bound (mentioned early in the introduction).

\begin{proposition}
  \label{p:isomorphic_forms}
  Assume that $k \geq 3$, $M$ is a compact orientable PL $2k$-manifold,
  and $M'$ is a compact $(k-1)$-connected orientable PL $2k$-manifold such that $M$ and $M'$ have isomorphic
intersection forms over the integers. If $\Delta_n^{(k)}$ (topologically) almost embeds into $M$,
then $\Delta_n^{(k)}$ PL embeds into $M'$. 
\end{proposition}

\begin{proof}
  Given an embedding of $\Delta_n^{(k)}$ to $M$,
  Theorem~\ref{t:obstruction} implies
  that there is a homomorphism $\psi \colon C_k(\Delta_n^{(k)}; \Z) \to H_k(M;
  \Z)$ such that $[\omega_\psi] - \oo(\Delta_n^{(k)}) = 0$. As the intersection
  forms of $M$ and $M'$ are isomorphic, there is also a homomorphism $\psi'
  \colon C_k(\Delta_n^{(k)}; \Z) \to H_k(M'; \Z)$ such that $[\omega_{\psi'}] -
  \oo(\Delta_n^{(k)}) = 0$. Therefore, we get the required PL embedding into $M'$ from
  Theorem~\ref{t:complete}.
\end{proof}

The intersection form on $M$ is in particular very simple if $k$ is odd
and $M$ is closed. The former property implies that the form is 
antisymmetric; the latter property implies that it is unimodular (after
factoring out the torsion)~\cite[Subsection~2.7]{prasolov07}.
Therefore for a suitable choice of the basis of
$H_k(M;\Z)$ (after factoring out the torsion) it can be represented by a block-diagonal matrix where each block is of the form 
$\begin{pmatrix}
0 & 1 \\
-1 & 0 \\
\end{pmatrix}$; this is a simple exercise (and probably a well known fact).
The
explicit reference containing the proof we were able to find are the online
lecture notes~\cite[Claim~2.1, Lecture~7]{morgan18_ln}. On the other hand, the
block diagonal matrix with $b$ blocks $\begin{pmatrix}
  0 & 1 \\
  -1 & 0 \\
\end{pmatrix}$ is the matrix of the intersection form of the connected sum of
$b$ copies of $S^k \times S^k$, which is $(k-1)$-connected. If we take this connected sum as $M'$, Proposition~\ref{p:isomorphic_forms} gives the following corollary.

\begin{corollary}
\label{c:connected_sum}
  Assume that $k \geq 3$ is odd and assume that $\Delta_n^{(k)}$
  (topologically) almost embeds into a closed orientable PL $2k$-manifold $M$, then
  $\Delta_n^{(k)}$ PL embeds into the connected sum 
  $(S^k \times S^k) \# \cdots \# (S^k \times S^k)$ of $\beta_k(M;\Z)$ copies of $S^k
  \times S^k$.\qed
\end{corollary}

In other words, Corollary~\ref{c:connected_sum} says that it we want to
solve the orientable variant of the K\"{u}hnel problem for $k$ odd, it is
sufficient to solve it in the very special cases for $M = (S^k \times S^k) \# \cdots
\# (S^k \times S^k)$.

\subsection{Computational aspects}
Part of our motivation for introducing the
obstruction for embeddability of $K$ into $M$ was to understand an analogue of
algorithmic embeddability question from~\cite{matousek-tancer-wagner11}, when
the target space is $M$ (instead of Euclidean space as
in~\cite{matousek-tancer-wagner11}). For this, let $\emb(k,M)$, for fixed
$k$ and $M$ denote the computational problem which asks whether a $k$-complex $K$ on
input is embeddable into $M$.

\begin{question}
\label{q:e_decidable}
  For which $2k$-manifolds is $\emb(k,M)$ decidable?
\end{question}

This problem of course makes sense even without the 
assumption that $\dim M = 2k$ but we will stay in the world of $2k$-manifolds as
this is the first nontrivial case. 
As mentioned early in this section, $\emb(k, \R^{2k})$ is decidable, even
polynomial time solvable, for $k \neq 2$. Also, if $k = 1$ and $M$ is an arbitrary (closed) surface,
then $\emb(1,M)$ is decidable, even linear time solvable~\cite{mohar99,
kawarabayashi-mohar-reed08}. If $k =2$, decidability of $\emb(k, \R^{2k})$ is
unknown.

For $k\geq 3$, our approach allows us to reformulate each instance of this
problem as a certain, very special, system of quadratic Diophantine equations.
This follows from
Theorems~\ref{t:obstruction},~\ref{t:complete},~and~\ref{t:quadratic} (stated
later on).
 Unfortunately, it is in general undecidable to determine whether a system of quadratic equations over integers 
has a solution~\cite{matijasevic70}. However, the system of
equations coming from Theorem~\ref{t:quadratic} is somewhat special and we suspect that it can be solved algorithmically for sufficiently nice $M$. In any
case, the reformulation of Question~\ref{q:e_decidable} via
Theorem~\ref{t:quadratic} allows to try new tools when answering
Question~\ref{q:e_decidable}.

On the other hand, if we consider the same system of quadratic equations over
$\Z_2$, then solvability of such a system is decidable (in worst
case by trying all options). 
This reflects in decidability stated in the following theorem.
The properties of maps stated in the theorem are 
generalizations of even drawings and independently even drawings of
graphs~\cite{pach-toth04,fulek-kyncl-malinovic-palvolgyi15,
fulek-kyncl19}.

\begin{theorem}
\label{t:Z2_computable}
  Let us assume that $k \geq 3$ and $M$ is a compact $(k-1)$-connected PL manifold. 
  Then,
  it is algorithmically decidable to determine whether a given $k$-complex $K$
  admits
  \begin{enumerate}[$(i)$]  
\item
  a general position map $f\colon |K|
  \to M$ such that whenever $\sigma$ and $\tau$ are disjoint $k$-simplices of
  $K$, then $f(\sigma)$ and $f(\tau)$ intersect an even number of times;
\item

  a general position map $f\colon |K|
      \to M$ such that whenever $\sigma$ and $\tau$ are $k$-simplices
      of
        $K$, then $f(\sigma)$ and $f(\tau)$ intersect an even number of times.
  \end{enumerate}
\end{theorem}

Finally, if $M$ is compact PL and simply connected then it can be efficiently decided
whether a given map $|K| \to M$ is homotopic to an embedding. For details we
refer to Remark~\ref{r:homotopy_computable}. However, this remark is not very
new; this is essentially just Johnson's~\cite{johnson02} description of the
obstruction, though not stated this way explicitly. We add this remark for completeness.

\subparagraph{Organization.} In Section~\ref{s:prelim} we properly introduce
the intersection number and 
the intersection form. Then,
Theorem~\ref{t:obstruction} is proved in Section~\ref{s:transfer};
Theorem~\ref{t:kuhnel_improved} is proved in
Section~\ref{s:kuhnel}; and 
Theorems~\ref{t:complete} and~\ref{t:Z2_computable} are 
proved in
Section~\ref{s:complete}.
In Section~\ref{s:problems} we mention a few open
problems. 

\section{Preliminaries}
\label{s:prelim}
Throughout the paper, we work in the PL-category. In particular, all maps and
manifolds are PL, unless stated otherwise. Simplicial complexes are geometric
simplicial complexes, that is, triangulations of polyhedra as
in~\cite{rourke-sanderson72}.
We assume that $k \geq 1$ is an integer,
and $R$ is either the ring $\Z$ of integers or $\Z_2$. We assume that $M$ is $R$-orientable
compact (possibly with boundary) $2k$-manifold, unless explicitly stated otherwise.
($\Z$-orientability is the standard orientability,
$\Z_2$-orientability is vacuous.) In sequel `oriented' stays for $R$-oriented
and all orientation considerations should be skipped if $R = \Z_2$.
We also assume that $K$ is $k$-complex in which each simplex has a fixed
orientation but we
do not require any compatibility conditions for orientations of different
simplices. By $K^{(k-1)}$ we denote the $(k-1)$-skeleton of $K$.
The closed interval
$[0,1]$ is denoted $I$.

\subsection{Intersection number}\label{subs:intersection}
In the definitions of general position and intersection number below
apart from our conventions on $M$, we also allow $M= \R^{2k}$. This fills the
postponed details from Subsection~\ref{ss:obstruction}.

\subparagraph{General position.}
Let $f \colon |K| \to M$ be a map. We say that $f$ is a \emph{general position
map} if $f|_{|K^{(k-1)}|}$ is injective; there are only finitely many $x$ with more than one preimage; each such $x$ has exactly two preimages, which both lie in
$|K|\setminus|K^{(k-1)}|$, and the crossing of $f$ at $x$ is transversal. In
addition, if $M$ has nonempty boundary, we assume that\footnote{It would be
perhaps more natural to assume $f(|K|) \subseteq M\setminus \partial M$ for a
general position map. However, allowing the nonempty intersection of $\partial
M$ and the image of $K^{(k-1)}$ will be useful in one of the proofs.}
$f(|K|\setminus|K^{(k-1)}|) \subseteq M\setminus \partial M$. We will sometimes
need to perturb a map $f$ to a general position map $f'$ by a homotopy with a
support in an arbitrarily close neighborhood of $f(|K|)$. In such case we mean
to use Lemma~4.8 of~\cite{hudson69}.

Sometimes, we will need a mutually general position of two maps $f \colon |K| \to M$, $f' \colon |K'| \to M$ where
$K'$ is another $k$-complex. This will be equivalent with requiring that 
$f \sqcup f' \colon |K| \sqcup |K'| \to M$ is a general position map, where
`$\sqcup$' stands for disjoint union.

\subparagraph{Intersection number.} Let $f\colon |K| \to M$ and
$f' \colon |K'| \to M$ be maps. Let $\sigma \in K$, $\tau \in K'$ be two
$k$-simplices such that $f|_{\sigma} \sqcup f'|_{\tau}$ is  in general
position. Let $x \in M$ be an intersection point of $f(\sigma)$ and
$f'(\tau)$, 
that is, $x = f(y) = f'(y')$ for some $y \in
\sigma$ and $y' \in \tau$.
By general position, the intersection is transversal and
$y$ is in the interior of $\sigma$ and $y'$ is in the interior of $\tau$.
By $\sgn_{f, f'}(x)$ we denote the
\emph{sign} of this intersection: 

If $R = \Z_2$, then $\sgn_{f, f'}(x) =
1$.

If $R = \Z$, because the intersection of $f(\sigma)$ and $f'(\tau)$
is transversal at $x$, there is a neighborhood $N(x)$ of $x$ in $M$ and an
orientation preserving
PL-embedding $g \colon N(x) \to \R^{2k}$ such that both $g(f(\sigma) \cap
N(x))$ and $g(f'(\tau) \cap N(x))$ are flat. 
Considering the orientations of $\sigma$ and $\tau$ as a choice of positively
oriented bases, this gives positively oriented bases of (affine spans of) $g(f(\sigma) \cap
N(x))$ and $g(f'(\tau) \cap N(x))$. Then by concatenation, taking a positively
oriented basis of $g(f(\sigma) \cap N(x))$ first, we get an orientation of
$g(N(x))$.
We set $\sgn_{f, f'}(x) = 1$ if this orientation agrees with the
orientation of $M$ (after applying $g$) and $-1$ otherwise; see
Figure~\ref{f:sign} for an example. It turns out that $\sgn_{f, f'}(x) = (-1)^k\sgn_{f', f}(x)$.

\begin{figure}
\begin{center}
  \includegraphics{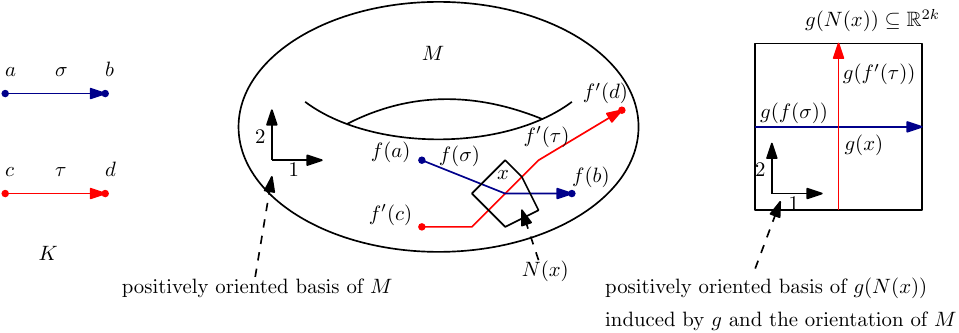}
\end{center}
  \caption{An example with $\sgn_{f,f'}(x) = +1$. Indeed,
  the basis $(g(f(\sigma) \cap N(x)), g(f'(\tau) \cap N(x)))$ gives the
  same orientation as the positively oriented basis of $g(N(x))$; in the
  picture $g(f(\sigma) \cap N(x))$ is simplified to $g(f(\sigma))$ and
  $g(f'(\tau) \cap N(x))$ to $g(f'(\tau))$.}
\label{f:sign}
\end{figure}

Next, the \emph{intersection number} of $f(\sigma)$ and
$f'(\tau)$ is defined as
\begin{equation}
\label{e:in}
  f(\sigma)\cdot f'(\tau) := \sum_x \sgn_{f, f'}(x)
\end{equation}
where the sum is over all $x$ obtained as intersection points of $f(\sigma)$
and $f'(\tau)$. Consequently,
\begin{equation}
\label{e:skew}
  f(\sigma)\cdot f'(\tau) = (-1)^k f'(\tau)\cdot f(\sigma).
\end{equation}

\subsection{Intersection form.}
\label{ss:if}

By $\Omega \colon H_k(M; R) \times H_k(M; R) \to R$
we denote the intersection form. Intuitively, given two cycles $z_1, z_2 \in
Z_k(M; R)$ in general position, the value $\Omega([z_1], [z_2])$ counts the
intersection number of these two cycles, which could be defined similarly as
for general position maps. 

\paragraph{Intersection form for closed manifolds.} Now, we temporarily assume that $M$ is closed.
In this case we refer to~\cite[Chapter 2, \S 2.7]{prasolov07} for precise definition; we use the dual form
$f^*$ in~\cite{prasolov07}. However, if $R = \Z$, we assume
that $\Omega$ is also defined on the torsion part of $H_k(M; R)$ and it
evaluates to $0$ there. (Prasolov~\cite{prasolov07} points out that the form
vanishes on the torsion part and he factors out the
torsion---then the form is nondegenerate.)
We will use the following properties of the intersection form:
\begin{enumerate}[$(i)$]
\item $\Omega$ is a bilinear form.
\item $\Omega$ is alternately-symmetric, that is, $\Omega(a,b) = (-1)^k \Omega(b,a)$.
\item $\Omega$ evaluates to $0$ on the torsion part of $H_k(M;R)$ if $R = \Z$.
\item Let $f\colon |K| \to M$ and $f' \colon |K'| \to M$ be maps
  such that $f \sqcup f'$ is in general position. Let $z \in Z_k(|K|;R), z' \in
  Z_k(|K'|;R)$, $z = \sum n_i
\sigma_i$, $z' = \sum n'_i \sigma'_i$ be two $k$-cycles,
where $n_i, n'_i \in R$ and $\sigma_i, \sigma'_i$
are all $k$-simplices of $K$ and $K'$ respectively. Then
\begin{equation}
\label{e:form}
\Omega(f_*([z]), f'_*([z'])) =
\sum\limits_{i,j} n_i n'_j f(\sigma_i)\cdot f'(\sigma'_j).
\end{equation}
\end{enumerate}

Property (i) follows immediately from the definition and (ii) is the contents of
Theorem 2.17(b) in~\cite{prasolov07}; (iii) is due to our convention. Finally,
(iv) comes from the definition of the intersection number in~\cite[Chapter 1,
\S 5.3]{prasolov07}. For getting formula~\eqref{e:form} we need that $z$ and
$z'$ are cycles in mutually dual cell decompositions of $M$ but this can be achieved
by considering sufficiently fine subdivision of $M$ and a perturbation of $z'$.
For other properties of the intersection form, we also refer
to~\cite{manifold_atlas_intersection_form}.

\paragraph{Intersection form for manifolds with nonempty boundary.}
As Prasolov points out~\cite[Remark below Thm.~2.17]{prasolov07}, the intersection form can be
also defined for manifolds with nonempty boundary. 
However, one has to be a bit careful because there are two natural ways how to do it, and in the case of manifolds with boundaries these two definitions are non-equivalent.
Here we provide a definition for which formula~\eqref{e:form} remains true.

Therefore now we assume that $M$ is compact with nonempty boundary. Let $M'$
be the double of $M$. (We take two copies of $M$ and we glue them together along their common
boundary). Let $\Omega_{M'}$ be the intersection form on $M'$ and we aim to
define the intersection form $\Omega = \Omega_M$ on $M$. Let $h_1, h_2$ be two
homology classes in $H_k(M; R)$. Let $z_1, z_2$ be $k$-cycles with $[z_1]_M =
h_1$ and $[z_2]_M=h_2$ where the subscript $M$ indicates that their homology class
is taken in $M$. (Analogously, we use subscript $M'$ if the homology class is
taken in $M'$). We define
\[
  \Omega_M(h_1, h_2) := \Omega_{M'}([z_1]_{M'}, [z_2]_{M'}).
\]
We note that $\Omega_M(h_1, h_2)$ is well defined because if $z$ and $z'$ are
homologous in $M$, then they are homologous in $M'$ as well. We also remark
that all the properties  (i), (ii), (iii), (iv) remain true for $M$ with
nonempty boundary---this can be easily checked because $M'$ satisfies them.

\section{Van Kampen obstruction in a manifold}
\label{s:transfer}

Throughout this section we have the same assumptions as in the beginning
of Section~\ref{s:prelim} regarding the notation $K$, $M$ and $R$.

Recall from the introduction that $\tilde K$ denotes the deleted product of $K$;
$C^{2k}_{\alts}(\tilde K; R)$ is the group of alternately-symmetric cochains;
$H^{2k}_{\alts}(\tilde K; R)$ is the corresponding cohomology group; and
$\oo(K) \in H^{2k}_{\alts}(\tilde K; R)$ is the van Kampen obstruction. Let us
also recall that $\oo(K) = [\vartheta_f]$ where $\vartheta_f$ is the
intersection cochain of an arbitrary general position map $f \colon |K| \to \R^{2k}$.

We generalize the intersecton cochain to maps with codomain $M$: Given a general position map $f \colon |K| \to M$ we define the
\emph{intersection cochain} for $f$ as $\vartheta_f \in C^{2k}_{\alts}(\tilde
K; R)$ via
$$
\vartheta_f(\sigma \times \tau) = f(\sigma) \cdot f(\tau).
$$
It follows from~\eqref{e:skew} that $\vartheta_f$ is alternately-symmetric as
required. We also define the \emph{van Kampen obstruction of the homotopy class
of $f$} as the cohomology class $\oo_f(K) := [\vartheta_f] \in
H^{2k}_{\alts}(\tilde K; R)$.

\begin{lemma}
\label{l:homotopic}
  Let $f, f'\colon |K| \to M$ be homotopic general position maps. Then
  $[\vartheta_f] = [\vartheta_{f'}]$. Equivalently, $\oo_f(K) := \oo_{f'}(K)$.
\end{lemma}

The proof of Lemma~\ref{l:homotopic} is given in~\cite[Lemma~3.5]{shapiro57}
and reproduced, e.g., in~\cite[Lemma~1]{freedman-krushkal-teichner94} in the
case that $M = \R^{2k}$. The proof is based on an existence of a homotopy
between $f$ and $f'$ and can be used essentially in verbatim in our setting.

\begin{lemma}
\label{l:f'}
  Let $f\colon |K| \to M$ be a general position map such that the restriction of
  $f$ to $|K^{(k-1)}|$ is nullhomotopic. Then there is a PL $2k$-ball $B$ in $M$ and a general position map $f' \colon
  |K| \to \overline{M \setminus B}$ homotopic to $f$ (in $M$) such that
   $f'(|K^{(k-1)}|) \subseteq \partial B$.
\end{lemma}

\begin{proof}
Because the restriction of $f$ to $|K^{(k-1)}|$ is nullhomotopic, by the homotopy
  extension property~\cite[Proposition~0.16]{hatcher01} there is $f'' \colon
  |K| \to M$ homotopic to $f$ such that the restriction of $f''$ to $|K^{(k-1)}|$ is
  constant. Let $x = f''(|K^{(k-1)}|)$ and let $B$ be $2k$-ball such that $x \in
  \partial B$. By further homotopy, we can get $f''' \colon |K| \to M$ such
  that $f'''$ restricted to $|K^{(k-1)}|$ is in general position, and $f'''(|K^{(k-1)}|)
  \subseteq \partial B$. (We first perform the homotopy on $|K^{(k-1)}|$ and then we use the
  homotopy extension property again.) Finally, by next homotopy fixed on $|K^{(k-1)}|$
  we push the image of $|K| \setminus |K^{(k-1)}|$ outside $B$ so that the resulting
  map is in general position, obtaining the required $f'$.
\end{proof}

Now the shift of $f$ to $f'$ in the previous lemma allows us to easily compare
the intersection cochain $\vartheta_f'$ with the intersection cochain $\vartheta_g$
of another map $g$ which is fully inside $B$. The advantage of using $g$ is
that this is essentially the case in $\R^{2k}$. Therefore
let $B \subseteq M$ be a $2k$-ball in $M$.
Let $f'\colon |K| \to \overline{M \setminus B}$ and $g\colon |K| \to B$ be two
general position maps such that $f'|_{|K^{(k-1)}|} = g|_{|K^{(k-1)}|}$. (In particular $f'(K^{(k-1)}) = g(K^{(k-1)}) \subseteq \partial B$.)

Now for a $k$-simplex $\sigma \in K$ let $z_\sigma$ be the (singular)
$k$-cycle $f'(\sigma) -
g(\sigma)$.
We also define $\omega_{f',g} \in C^{2k}_{\alts}(\tilde K; R)$ via
\begin{equation}
\label{e:omega_fg}
  \omega_{f',g}(\sigma \times \tau) := \Omega([z_\sigma], [z_\tau]).
\end{equation}

\begin{lemma}
\label{l:omega}
  $\vartheta_{f'} = \omega_{f',g} - \vartheta_{g}$.
\end{lemma}

\begin{proof}
 For $\sigma \times \tau \in \tilde K$ we have
 \[\omega_{f',g}(\sigma \times \tau) = \Omega([z_\sigma], [z_\tau]) =
  f'(\sigma)\cdot f'(\tau) + g(\sigma)\cdot g(\tau) = \vartheta_{f'}(\sigma \times
\tau) + \vartheta_{g}(\sigma \times \tau).\]
The second equality follows from the fact that $f'(\sigma) \cap g(\tau) =
  g(\sigma) \cap f'(\tau) = \emptyset$ and from~\eqref{e:form}.
\end{proof}

\begin{proposition}
\label{p:transfer}
Let $f \colon |K| \to M$ be a general position map with $\vartheta_f = 0$.
  Assume that the restriction of $f$ to $|K^{(k-1)}|$ is nullhomotopic.
  Then there is a homomorphism $\psi \colon
  C_k(K; R) \to H_k(M; R)$ such that $[\omega_\psi] - \oo(K)$ is
  trivial.
\end{proposition}

\begin{proof}
  Let $f' \colon |K| \to \overline{M \setminus B}$ be the map obtained from
  Lemma~\ref{l:f'}. Take an arbitrary general position map $g\colon
    |K| \to B$ which coincides with $f'$ on $\partial B$ and define $z_\sigma$
    and 
      $\omega_{f',g}$ as above. By Lemma~\ref{l:omega} and
      Lemma~\ref{l:homotopic}, we get
        $[\omega_{f',g} - \vartheta_{g}] = [\vartheta_{f'}] = [\vartheta_f] =
	[0] = 0.$ Let us define $\psi(\sigma)$ to be the homology class of
	$z_\sigma$ in $H_k(M; R)$. Then, according to the
	earlier definition of $\omega_{\psi}$, we get $\omega_{\psi} =
	\omega_{f',g}$ and $\oo(K) = [\vartheta_g]$. Therefore
	$[\omega_\psi] - \oo(K) = [\omega_{f',g} - \vartheta_{g}] = 0$.
\end{proof}

Theorem~\ref{t:obstruction} is an immediate consequence.

\begin{proof}[Proof of Theorem~\ref{t:obstruction}]
In this proof we use both topological and PL maps therefore we carefully
distinguish them.
Let $f$ be a topological almost embedding from the statement of
Theorem~\ref{t:obstruction}. By a small perturbation
(cf.~\cite[Lemma~4.8]{hudson69}) we can assume that $f$ is a general position 
PL map and still
an almost embedding (if $f(\sigma)$ and $f(\tau)$ have a positive distance
before the perturbation in some metric inducing topology of $M$, then they have
positive distance also after a sufficiently small perturbation). In particular
$\vartheta_f = 0$. Now we can use
  Proposition~\ref{p:transfer}.
\end{proof}

\subparagraph{Finger moves.} For further applications it is useful to describe the zero cohomology
class in $H^{2k}_{\alts}(\tilde K; R)$ explicitly. This we will do now via so
called elementary cochains/finger moves.

Following~\cite{freedman-krushkal-teichner94} let
us define elementary cochains as follows: Given a $(k-1)$-simplex $\eta$ and
$k$-simplex $\mu$ with $\eta \cap \mu = \emptyset$, let $\varepsilon_{\eta,\mu} \in C^{2k-1}_{\sym}(\tilde K; R)$ be
given by $\varepsilon_{\eta,\mu}(\eta \times \mu) = \varepsilon_{\eta,\mu}(\mu \times
\eta) = 1$ while it evaluates to zero on remaining $2k-1$ cells. These cochains
generate $C^{2k-1}_{\sym}(\tilde K; R)$. Therefore $\delta \varepsilon_{\eta,\mu}$
generate $\im \delta_{2k-1}$; compare with~\eqref{e:coboundary_operators}. 
In other words, for $\xi \in C^{2k}_{\alts}(\tilde K; R)$ and
its cohomology class $[\xi] \in H^{2k}_{\alts}(\tilde K; R)$ we get 

\begin{equation}
\label{e:sum_finger_moves}
  [\xi] = 0 \hskip1cm \hbox{if and only if} \hskip1cm \xi = \sum_{\eta, \mu} n_{\eta,\mu} \delta \varepsilon_{\eta,\mu}
\end{equation}
where the sum is over all pairs $\eta$, $\mu$ as above and $n_{\eta,\mu} \in
R$.

Geometrically, the cochains $\delta \varepsilon_{\eta,\mu}$ correspond to so
called finger moves; see Figure~\ref{f:fm}. If we pull a finger from $\mu$
towards $\eta$ as suggested on Figure~\ref{f:fm}, the intersection cochain
changes by $\pm \delta \varepsilon_{\eta,\mu}$ (both signs are achievable
depending on how is the finger pulled). For a slightly more details, we refer
to~\cite{freedman-krushkal-teichner94}.  

\begin{figure}
\begin{center}
  \includegraphics{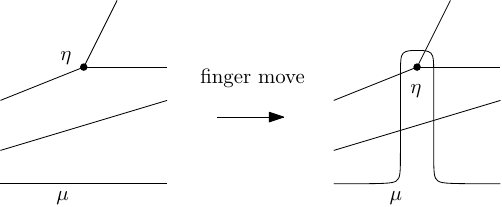}
\end{center}
  \caption{Geometric finger move inducing the change on the intersection
  cochain by $\pm \delta \varepsilon_{\eta,\mu}$.}
\label{f:fm}
\end{figure}

\subparagraph{System of quadratic equations.}
Our next aim is to describe an existence of almost embedding via solvability of
a certain system of quadratic equations. 

Let $\eta, \mu \in K$ be a $(k-1)$-simplex and $k$-simplex respectively and
assume that $\eta$ and $\mu$ are disjoint. For every such pair we define a
variable $x_{\eta,\mu}$. 

Next we need to distinguish whether $R = \Z$ or $R = \Z_2$. If $R = \Z$, 
assume that $H_k(M; \Z) \cong \Z^b \oplus
T_k(M; \Z)$ where $T_k(M; \Z)$ is the torsion. Let $\pi\colon H_k(M; \Z) \to \Z^b$
be the homomorphism obtained from the isomorphism above after factoring out the
torsion. If $R = \Z_2$, then $H_k(M; \Z_2) \cong \Z_2^b$ for some $b$ and we
take an arbitrary isomorphism $\pi\colon H_k(M; \Z_2) \to \Z_2^b$.

Let $\AA_\Omega \in R^{b \times b}$ be the matrix of $\Omega$, that is, for
every $h,h' \in H_k(M;R)$ we have
$\Omega(h,h') = \pi(h)^{T}\AA_{\Omega}\pi(h')$.
For every $k$-simplex $\sigma$ and every $i \in \{1, \dots, b\}$ we define an integer variable
$y_{\sigma}^i$ and we set $\yy_{\sigma} := (y_{\sigma}^1, \dots, y_{\sigma}^b)$.
  Let $\vartheta_g$ be any fixed intersection cochain in the class $\oo_K$; an explicit choice is described in~\cite[App. D]{matousek-tancer-wagner11}.

Now consider a system of quadratic equations with variables
$x_{\eta,\mu}$ and $\yy_{\sigma}$ over $R$ given by the following
equation for each pair $\{\sigma, \tau\}$ of disjoint $k$-simplices.
\begin{equation}
\label{e:quadratic}
\sum_{\eta,\mu} x_{\eta,\mu}\delta\varepsilon_{\eta,\mu}(\sigma \times \tau) +
  \yy_{\sigma}^T \AA_{\Omega} \yy_{\tau}  = \vartheta_g(\sigma \times
  \tau).
\end{equation}
We remark that swapping $\sigma$ and $\tau$ gives the same equation as both
sides are alternately-symmetric. 

\begin{theorem}
\label{t:quadratic}
  Let $M$ be a compact $R$-orientable PL $2k$-manifold. Then there is a
homomorphism $\psi \colon C_k(K; R) \to H_k(M; R)$ such that $[\omega_\psi] -
  \oo(K)$ is
  trivial (considered over $R$) if and only if the system of
  equations~\eqref{e:quadratic} has a solution in $R$.
\end{theorem}

\begin{proof} First assume that $[\omega_\psi] - \oo(K)$ is
  trivial, 
  hence $\omega_\psi$ and $\vartheta_g$ differ by a linear combination of
  cochains $\delta \varepsilon_{\eta,\mu}$
  by~\eqref{e:sum_finger_moves}.
  Thus,
  there are $x_{\eta,\mu} \in R$, one for each $\delta
  \varepsilon_{\eta,\mu}$,
  such
  that for every $\sigma \times \tau \in \tilde K$ we get $\omega_\psi(\sigma
  \times \tau) - \vartheta_g(\sigma \times \tau) = \sum_{\eta, \mu}
  x_{\eta,\mu} \delta\varepsilon_{\eta,\mu}(\sigma \times \tau)$. We also set
  $\yy_\sigma$ as $\pi(\psi(\sigma))$, then $\omega_\psi(\sigma \times \tau) =
  \yy_{\sigma}^T \AA_{\Omega} \yy_{\tau}$. By rearranging and swapping the
  signs at all $x_{\eta,\mu}$ we get a solution of~\eqref{e:quadratic}.

  Now assume that we have a solution of~\eqref{e:quadratic}. For a $k$-simplex
  $\sigma \in K$, we define $\psi(\sigma)$ as an arbitrary element in $\pi^{-1}(\yy_\sigma)$ 
  and we extend $\psi$ to a
  homomorphism from $C_k(K; \Z)$ to $H_k(M; \Z)$. We get $\omega_\psi(\sigma
  \times \tau) = \yy_{\sigma}^T \AA_{\Omega} \yy_{\tau}$. Therefore,
  from~\eqref{e:quadratic}, we get that $\omega_\psi - \vartheta_g$ is a
  linear combination of cochains $\delta\varepsilon_{\eta,\mu}$
  by~\eqref{e:sum_finger_moves}.
  This gives that $\omega_\psi$ and $\vartheta_g$ belong  to the same
  cohomology class of $H^{2k}_{\alts}(\tilde K; R)$; that is, $[\omega_\psi] - \oo(K)$ is trivial.
\end{proof}

\section{K\"{u}hnel question}
\label{s:kuhnel}

  In this section we prove Theorem~\ref{t:kuhnel_improved}. However; first we sketch a proof of Theorem~\ref{t:kuhnel_improved} with slightly weaker
bound $n \leq 2k+1 + (k+2)\beta_k(M;\Z_2)$. For a full proof of
Theorem~\ref{t:kuhnel_improved}, it is not necessary to follow this sketch.
However, it may help to understand why do we prepare some auxiliary claims.

  \begin{proof}[Sketch of Theorem~\ref{t:kuhnel_improved} with a weaker bound $n
    \leq
    2k+1 + (k+2)\beta(M;\Z_2)$.]
  Let us assume that $K = \Delta^{(k)}_n$ almost embeds into $M$. By
  Theorem~\ref{t:obstruction}, there is $\psi\colon C_k(K; \Z_2) \to H_k(M;\Z_2)$
  such that $[\omega_\psi] + \oo(K) = 0$. Consider an induced
  subcomplex $J$ of $K$ on $2k+3$ vertices (if $n \geq 2k+2$, otherwise we are
  done). It is well known that the $\Z_2$ van
    Kampen obstruction of this complex is nonzero. The formula $[\omega_\psi]
    + \oo(K) = 0$ therefore implies that
  $\omega_\psi$ is also nontrivial when restricted to $J$. In particular, we
  will deduce that there is a pair of cycles $z, z' \in Z_k(J;\Z_2)$ such that
  $\Omega(\psi(z), \psi(z')) = 1$. Because $Z_k(J;\Z_2)$ is generated by
  boundaries of $(k+1)$-simplices, we can assume that $z = \partial \kappa, z' =
    \partial \kappa'$
  where $\kappa$ and $\kappa'$ are $(k+1)$-simplices in $J$. On
    the other hand we will also show that $\Omega(\psi(\partial \kappa), \psi(\partial
    \lambda)) = 0$ whenever $\lambda$ is $(k+1)$-simplex disjoint from $\kappa$.
    (In particular $\kappa'$ must share a vertex with $\kappa$ but we do not
    need it now.)

 Let $K'$ be the induced subcomplex of $K$ on all vertices of $K$ except the
  $k+2$ vertices of $\kappa$. From the conditions $\Omega(\psi(\partial \kappa),
    \psi(\partial \kappa')) = 1$ and $\Omega(\psi(\partial \kappa),
        \psi(\partial \lambda)) = 0$ (for $\lambda$ disjoint with $\kappa$) it follows that when considering the restriction of $\Omega$ to $K'$, the
  rank of $\Omega$ is decreased at least by $1$. We repeat this procedure as
  long as the currently considered subcomplex has at least $2k+3$ vertices.
  In the beginning, the rank of $\Omega$ is
  at most $\beta_k(M;\Z_2)$, thus we may perform at most $\beta_k(M;\Z_2)$ such
    removals of $k+2$ vertices. This shows that the number of vertices $K$
    (which is $n+1$) is at  most $\beta_k(M;\Z_2)(k+2) + 2k+2$ because after all removals we cannot have
  a subcomplex on $2k+3$ vertices or more.
\end{proof}

It is not too difficult to fill the details in the sketch above. However, as our bounds 
in Theorem~\ref{t:kuhnel_improved} indicate, in a proof of
Theorem~\ref{t:kuhnel_improved}(i) we want to remove only $k+1$ vertices
instead of $k+2$ in each step; on the other hand in a proof of
Theorem~\ref{t:kuhnel_improved}(ii), we will remove $k+2$ vertices but we will
decrease the rank of the form by $2$. These improvements require a more subtle
analysis.

Throughout this section, let $k,n \geq 1$. We set $K := \Delta^{(k)}_{n}$ to be the
$k$-skeleton of an $n$-simplex while $M$ is a compact PL $2k$-manifold. We work only with $\Z_2$ coefficients, that is, we set $R =\Z_2$.
Note the notions
`alternately symmetric' and `symmetric' coincide over $\Z_2$. In particular, 
$\Omega$ is a
symmetric bilinear form over $\Z_2$ in this case. We also systematically replace `$\alts$' with `$\sym$' in expressions such as
$C^{2k}_{\alts}(\tilde K; \Z_2)$.

Now we state the two main ingredients for the proof of
Theorem~\ref{t:kuhnel_improved}. Given a vertex $v$ of $\Delta_n$ and a simplex
$\sigma \in \Delta_n$ not
containing $v$, by $\sigma * v$ we denote the join of $\sigma$ and $v$,
that is, the simplex formed by vertices of $\sigma$ and by $v$. Note that if
$\kappa$ is a $(k+1)$-simplex in $\Delta_n$, then $\partial \kappa$ belongs to
$K$.

\begin{proposition}\label{p:vanKampenForm} Assume that 
  $K$ almost embeds into $M$, then there is a symmetric bilinear form $\Lambda$ on
  $Z_k(K; \Z_2) = H_k(K; \Z_2)$ of rank at most $\beta_k(M;\Z_2)$ satisfying
  the following conditions.
  \begin{enumerate}[{\rm (C1)}]
\item
  Let $\kappa$, $\kappa'$ be disjoint $(k+1)$-simplices in $\Delta_{n}$.
      Then $\Lambda(\partial \kappa, \partial \kappa') = 0$.\label{it:disjoint}
    \item Let $J$ be an induced subcomplex of $K$ on $2k+3$ vertices (that is,
      $J$ is isomorphic to $\Delta^{(k)}_{2k+2}$). Then for every vertex $v$ of
      $J$ we get\label{it:kfive}
      \[    \sum_{\{\sigma, \tau\} \in P_{J,v}} \Lambda(\partial(\sigma *
	v),\partial(\tau * v)) = 1
      \]
      where $P_{J,v}$ is the set of all unordered pairs $\{\sigma, \tau\}$ of
      disjoint $k$-simplices in $J$ avoiding $v$.
\end{enumerate}

  In addition, if $\Omega(h,h) = 0$ for every $h \in H_k(M;\Z_2)$, then
  $\Lambda(z,z) = 0$ for every $z \in Z_k(K; \Z_2)$.
\end{proposition}

Relation between Proposition~\ref{p:vanKampenForm} and the preceding sketch is
the following: If $K$ embeds into $M$, then by Theorem~\ref{t:obstruction} there is $\psi\colon C_k(K;R) \to H_k(M;R)$ such
  that $[\omega_\psi] -  \oo(K) = 0$. We will take $\Lambda$ so that $\Lambda(z, z') =
\Omega(\psi(z), \psi(z'))$. Then the conditions (C\ref{it:disjoint}) and (C\ref{it:kfive}) 
verify
the conditions required in the sketch.
Although we do not need it, we note that the proof of Theorem~3
in~\cite{kyncl20} shows that 
the other implication can be partially reverted for $k=1$: If
(C\ref{it:disjoint}) and (C\ref{it:kfive}) are satisfied for
$\Lambda$ obtained from $\Omega$ and $\psi$ as above, then
$[\omega_\psi] + \oo(K)$ is trivial.

\begin{proposition}
\label{p:induction}
Assume that $\Lambda$ is a symmetric bilinear form on $Z_k(K; \Z_2)$
  satisfying conditions {\rm (C1)} and {\rm (C2)} (in
  Proposition~\ref{p:vanKampenForm}).

Then
\begin{gather}
   n\leq (2k+1) + (k+1)\operatorname{rank}\Lambda\quad\text{and}\label{e:general_Lambda}\\
   n\leq (2k+1) + \frac{(k+2)\operatorname{rank}\Lambda}{2}\quad\text{if $\Lambda(z,z)=0$ for all $z\in
   Z_k(K;\Z_2)$.}\label{e:symplectic_Lambda}
\end{gather}
\end{proposition}

Assuming the two propositions above, Theorem~\ref{t:kuhnel_improved} follows
immediately:

\begin{proof}[Proof of Theorem~\ref{t:kuhnel_improved}]
Assume that $K$ almost embeds into $M$. Let $\Lambda$ be the symmetric bilinear
  form on $Z_k(K; \Z_2)$ obtained from Proposition~\ref{p:vanKampenForm}.
  Because $\operatorname{rank}\Lambda\leq\beta_k(M;\Z_2)$, we immediately deduce Theorem~\ref{t:kuhnel_improved}(i) from~\eqref{e:general_Lambda}.
  If, in addition, $\Omega(h,h) = 0$ for every $h \in H_k(M; \Z_2)$, then
  $\Lambda(z,z) = 0$ for every $z \in Z_k(M;\Z_2)$ and we deduce
  Theorem~\ref{t:kuhnel_improved}(ii) from~\eqref{e:symplectic_Lambda}.
\end{proof}

Propositions~\ref{p:vanKampenForm} and~\ref{p:induction} are proved in
forthcoming subsections. 

In fact we conjecture that the bounds given by Proposition~\ref{p:induction} can
be improved to the K\"{uhnel} bounds:

\begin{conj}
\label{c:kuhnel_lambda}
Assume that $\Lambda$ is a symmetric bilinear form on $Z_k(K; \Z_2)$
  satisfying conditions {\rm (C1)} and {\rm (C2)} (in
  Proposition~\ref{p:vanKampenForm}).

Then
\[
  \binom{n-k-1}{k+1}\leq \binom{2k+1}{k+1}\operatorname{rank}\Lambda.
\]
\end{conj}

Proposition~\ref{p:vanKampenForm} and Conjecture~\ref{c:kuhnel_lambda}
together imply K\"{u}hnel's conjecture (for PL manifolds) in the same way as
Theorem~\ref{t:kuhnel_improved} is proved. In fact they imply even something
stronger. (It is not necessary to assume $(k-1)$-connectedness and the
conclusion holds for almost-embeddings.)

\paragraph{Computer-assisted bounds.} In our proof of Theorem~\ref{t:kuhnel_improved}, we do not use
Proposition~\ref{p:vanKampenForm} in full strength---at least for small
values the bounds can be improved:
If $b$ is odd, all non-degenerate symmetric bilinear forms on $\Z_2^b$ are
equivalent to the form with the identity matrix $\II_b$.
If $b=2c$, we furthermore have symplectic forms---forms equivalent to
\[\QQ_b=\begin{pmatrix}
0 & \II_c \\ \II_c & 0 \end{pmatrix}.\]
These are the forms satisfying $\Lambda(z,z) = 0$ for every
$z$.
The matrix of $\Lambda$ can hence be written as $\AA_\Lambda =
\BB^T \XX \BB$, where $\XX =\II_{b'}$ or $\XX=\QQ_{b'}$ for $b'=\operatorname{rank} \Lambda\leq\beta_k(M;\Z_2)$, and $\BB$ is a $b'\times  \dim Z_k(K;\Z_2)$ matrix over $\Z_2$.
Proposition~\ref{p:vanKampenForm} then translates into equations over $\Z_2$
which $\BB$ needs to satisfy. 
For small values these equations can be turned into a CNF formula and checked by modern SAT solvers,
preferably ones that support xor clauses, e.g. CryptoMiniSat~\cite{soos09}.
Using this technique we obtain computer assisted bounds in
Table~\ref{tab:small_values}.

\begin{table}
\begin{center}
\begin{tabular}{c|c|c|c|c}
  $k$ & $\beta$
  & max $n$, $\Lambda\sim \II$ & max $n$,
  $\Lambda$ symplectic & K\"uhnel's conjecture \\ \hline
  \multirow{4}{*} 1 & 1 & {\bf 5} (Prop~\ref{p:induction}) &- & 5\\
					& 2 & {\bf 5} & {\bf 6}
					(Prop~\ref{p:induction}) & 6 \\
					& 3 & {\bf 6} &- & 6 \\
					& 4 & {\bf 7} & {\bf 7} & 7\\ \hline
  \multirow{2}{*} 2 & 1 &  {\bf 8} (Prop~\ref{p:induction}) &-&8\\
					& 2 &  {\bf 8} & 7 & 9\\ \hline
  \multirow{1}{*} 3& 1  & $9\leq n\leq 10$ & - & 11\\ \hline
  \multirow{1}{*} 4 & 1 & {\bf 14} (Prop~\ref{p:induction}) & - & 14 \\ \hline
\end{tabular}

  \caption{The table gives
maximal $n$ for which there is a symmetric bilinear form $\Lambda$ on
  $Z_k(\Delta_n^{(k)},\Z_2)$ of rank $\beta$
  satisfying conditions (C1) and~(C2)
  of Proposition~\ref{p:vanKampenForm}, distinguishing whether $\Lambda \sim
\II$ or whether $\Lambda$ is symplectic. Via Proposition~\ref{p:vanKampenForm},
  if $\Delta_{n'}^{(k)}$ almost embeds into $M$ with $\beta_k(M, \Z_2) =
  \beta$ and the intersection form $\Omega \sim \Lambda$, then $n' \leq n$
  where $n$ is the value in the table.
  (Strictly speaking, this
  requires checking that $\Omega \sim \Lambda$ in the proof of
  Proposition~\ref{p:vanKampenForm}.)
\newline
  For the values in bold, there is a matching lower bound in a strong sense: There
  exists a $(k-1)$-connected closed manifold $M$ with $\beta_k(M, \Z_2) =
  \beta$ and the intersection form $\Omega \sim \Lambda$ such that
  $\Delta_n^{(k)}$ embeds into $M$. Indeed for $k=1$ and $\beta \leq 4$, the lower bounds is given by
  the Ringel-Youngs theorem~\cite{Ringel68,ringel74}.
For $k=2$ and $\beta_2=1$, there is Kühnel's $9$-point triangulation of $\C P^2$~\cite{Kuhnel1983}. Taking its two skeleton shows that $\Delta_{8}^{(2)}$ embeds into $\C P^2$.
For this reason $\Delta_{8}^{(2)}$ also embeds into $\C P^2\sharp \C P^2$ (for which $\beta_2=2$).
The case $k=4,\beta_k(M;\Z_2)$ corresponds to the $4$-skeleton of the $15$-point triangulation of the quaternionic projective plane~\cite{Gorodkov19}.
\newline
Last but not least, in the 5th column, we provide a bound that would follow
  from K\"{u}hnel's conjecture (which does not distinguish type of $\Lambda$).
  It is interesting to note that in the cases $(k,\beta) = (2,2), (3,1)$ our
  computer assisted bounds even beat K\"{u}hnel's conjecture. 
  }
\label{tab:small_values}
\end{center}
\end{table}

\subsection{Proof of Proposition~\ref{p:vanKampenForm}}
\label{ss:vanKampenForm}
Given $\psi$ as in
Theorem~\ref{t:obstruction}, as announced earlier, we define a symmetric
bilinear form $\Lambda$ on $Z_k(K; \Z_2)$ via $\Lambda(z,z') :=
\Omega(\psi(z),\psi(z'))$. We observe that the rank of $\Lambda$ is at most the
rank of $\Omega$ which is at most $\beta_k(M; \Z_2)$. We also observe that if
$\Omega(h,h) = 0$ for every $h \in H_k(M;\Z_2)$, then $\Lambda(z,z) = 0$ for
every $z \in Z_k(K;\Z_2)$. These are some of the conditions required on
$\Lambda$ in Proposition~\ref{p:vanKampenForm}.

For the proof of Proposition~\ref{p:vanKampenForm} we also need the following
lemma.
  
\begin{lemma}
  \label{l:independence_of_choice}
  Let $\Xi \in H^{2k}_{\sym}(\tilde K; \Z_2)$ be a cohomology class. Let $c$
  be a $2k$-chain in the ordinary chain group $C_{2k}(\tilde K; \Z_2)$ such
  that $\partial c$ is symmetric. (That is, $\sigma \times \eta$ and $\eta
  \times \sigma$ appear with the same coefficient in $\partial c$ for any $k$-simplex
  $\sigma$ and $(k-1)$-simplex $\eta$.) Then the value $\xi(c)$ is independent
  of the choice of the representative $\xi \in C^{2k}_{\sym}(\tilde K; \Z_2)$ with
  $[\xi] = \Xi$.
\end{lemma}

\begin{proof}
  Let $\xi, \xi' \in C^{2k}_{\sym}(\tilde K; \Z_2)$ be such that $[\xi] =
  [\xi']$. Let $\zeta = \xi - \xi'$. By the previous condition, $\zeta$ is
  cohomologically trivial, thus $\zeta = \delta \rho$ for some $\rho \in
  C^{2k-1}_{\sym}(\tilde K; \Z_2)$. Then we get
  \[
    \xi(c) - \xi'(c) = \zeta(c) = \delta \rho (c) = \rho(\partial c) = 0.
  \]
  The last equality above follows from the facts that $\rho(\sigma \times \eta)
  = \rho(\eta \times \sigma)$ for any $k$-simplex
  $\sigma$ and $(k-1)$-simplex $\eta$ and that $\partial c$ is symmetric.
  Thus we get $\xi(c) = \xi'(c)$ as required. 
\end{proof}

Now, let $\psi$ be a homomorphism as in Theorem~\ref{t:obstruction}. The
conclusion of Theorem~\ref{t:obstruction} is that $[\omega_\psi]$ and $\oo(K)$
are the same homology class. On the one hand $\omega_\psi$ is a representative of
this cohomology class; on the other hand an arbitrary intersection cochain
$\vartheta_g$ of a general position map $g\colon |K| \to \R^{2k}$ is a representative
as well. Consequently, Lemma~\ref{l:independence_of_choice} gives
\begin{equation}
\label{e:omega_theta_c}
  \omega_{\psi}(c) = \vartheta_g(c) \hskip3cm \hbox{for $c \in C_{2k}(\tilde K;
  \Z_2)$ such that $\partial c$ is symmetric.}
\end{equation}

Our strategy is that we will deduce  Proposition~\ref{p:vanKampenForm}
from~\eqref{e:omega_theta_c} by suitable choices of $c$ and $\vartheta_g$.

\begin{proof}[Proof of Proposition~\ref{p:vanKampenForm}(i)]
Because $K$ almost embeds into $M$, there is $\psi\colon C_k(K;\Z_2) \to
  H_k(M;\Z_2)$ such
  that $[\omega_\psi] = \oo(K)$ by Theorem~\ref{t:obstruction}.
  Take $\Lambda$ via $\Lambda(z,z') := \Omega(\psi(z), \psi(z'))$ as described above.

  Take the $2k$-chain $c \in C_{2k}(\tilde K; \Z_2)$ given by $c = \partial
  \kappa \times \partial \kappa'$. Note that $c$ is actually a cycle. Thus
  $\partial c = 0$ and $c$ satisfies the assumptions of
  Lemma~\ref{l:independence_of_choice}.
  Let $g \colon |K| \to \R^{2k}$ be an arbitrary general position map such that
$g(\partial \kappa)$ and $g(\partial \kappa')$ are disjoint. (This is possible
  because $\kappa$ and $\kappa'$ are disjoint.) Then we get
\[
  \Lambda(\partial \kappa, \partial \kappa') = \Omega(\psi(\partial
  \kappa),\psi(\partial(\kappa')) = \omega_\psi(c) = \vartheta_g(c)= 0 
\]
as required. The first equality is the definition of $\Lambda$; the second
  equality comes from the definition of $\omega_\psi$; the third equality is the contents
  of~\eqref{e:omega_theta_c}; 
  and the last equality
  follows from our choice of $g$ which implies that the intersection cochain
  $\vartheta_g$ is identically zero on $c$.
\end{proof}
For a proof of Proposition~\ref{p:vanKampenForm}(ii), we will need
another auxiliary
lemma which will be also reused in a proof of Proposition~\ref{p:induction}.
Given an induced subcomplex $J$ of $K$ on $2k +
3$ vertices let $P_J$ be the set of all unordered pairs $\{\sigma',\tau'\}$ of
  disjoint $k$-simplices in $J$. We also consider a chain $c_J \in C_{2k}(
  J; \Z_2) \subseteq C_{2k}(K; \Z_2)$ given by 
  \[c_J := \sum_{\substack{\{\sigma', \tau'\}
  \in P_J\\ \sigma' \prec \tau'}}  \sigma' \times \tau',\]
  where $\prec$ is an arbitrary fixed order on $k$-simplicies of $K$. We check
  that $c_j$ satisfies the assumptions of Lemma~\ref{l:independence_of_choice},
  that is, $\partial c_J$ is symmetric: Let $t$ be the isomorphism of
  $C_{2k}(K; \Z_2)$ swapping the coordinates. Then $c_J + t(c_J)$ is a cycle formed
  by all products $\sigma' \times \tau'$ over all ordered pairs $(\sigma',
  \tau')$ of disjoint $k$-simplices of $J$. Therefore $\partial c_J + \partial
  t(c_J) = 0$ which implies $\partial c_J = \partial
  t(c_J)$ (as we work over $\Z_2$). Consequently $\partial c_J$ is symmetric as
  required.

\begin{lemma}
\label{l:beta_omega_z_J_new}
  
  Let $\Lambda$ be a symmetric bilinear form on $Z_k(K;\Z_2)$. Let $J$ be
  an induced subcomplex of $K$ on $2k+3$ vertices.
  Then
  for arbitrary vertex $v$ of $J$ the value:
\begin{equation}
\label{e:sum_beta_omega_z_J}
  \sum_{\{\sigma,\tau\} \in P_{J,v}} \Lambda(\partial(\sigma *
v),\partial(\tau * v))
\end{equation}
is independent of the choice of $v$. 
In addition, if $\psi\colon C_k(K; \Z_2) \to H_k(M,\Z_2)$ is a homomorphism and
  $\Lambda(z,z') = \Omega(\psi(z),\psi(z'))$ for every $z, z' \in \Z_k(K;
  \Z_2)$, then the value of the sum~\eqref{e:sum_beta_omega_z_J} equals
  $\omega_\psi(c_J)$.
\end{lemma}

\begin{proof}
  First, we extend $\Lambda$ to a symmetric bilinear form $\Lambda'$ on
  $C_k(K;\Z_2)$. If we are in the `in addition' case that $\Lambda(z,z') =
  \Omega(\psi(z),\psi(z'))$, then we simply set $\Lambda'(\sigma, \tau) =
  \Omega(\psi(\sigma),\psi(\tau))$ for $k$-simplices $\sigma$ and $\tau$ of $J$.

  In any other case, $\Lambda'$ can be obtained in the following way: We pick a vertex $v_0 \in J$. Then for
  $k$-simplices $\sigma$ and $\tau$ in $J$ we set
$$\Lambda'(\sigma, \tau) = 
\begin{cases}
  0 & \hbox{if $\sigma$ or $\tau$ contains $v_0$}; \\
  \Lambda(\partial(\sigma * v_0),\partial(\tau * v_0)) &
  \hbox{otherwise}. \\
\end{cases}$$
  It is easy to check that $\Lambda(\partial(\sigma * v_0),\partial(\tau *
  v_0)) = \Lambda'(\partial(\sigma * v_0),\partial(\tau * v_0))$ if
  both $\sigma$ and $\tau$ avoid $v_0$. In addition, because the cycles
  $\partial(\sigma * v_0)$ for $\sigma$ avoiding $v_0$ generate $Z_k(K;
  \Z_2)$, we get $\Lambda(z,z') = \Lambda'(z,z')$ for any $z, z' \in Z_k(K;
  \Z_2)$.

Our aim is to show that 
\begin{equation}
  \label{e:RHS_independent}
  \sum_{\{\sigma,\tau\} \in P_{J,v}} \Lambda(\partial(\sigma *
   v),\partial(\tau * v)) =  \sum_{\{\sigma', \tau'\} \in
   P_J}\Lambda'(\sigma', \tau').
\end{equation}
  Because the right-hand side of~\eqref{e:RHS_independent} is independent of the
  choice of $v$, this will prove the first part of the lemma. If we are in the
  `in addition' case, then 
  we further get
  \[
  \sum_{\{\sigma', \tau'\} \in P_J}\Lambda'(\sigma', \tau') = 
  \sum_{\{\sigma', \tau'\} \in P_J}\Omega(\psi(\sigma'), \psi(\tau')) = 
  \omega_\psi(c_J)
  \]
where the last equality follows from the definitions of $\omega_\psi$ and
  $c_J$ (from $c_J$ we only need that for every $\{\sigma', \tau'\} \in P_J$
  exactly one of the products $\sigma' \times \tau'$ and $\tau' \times \sigma'$
  appears with coefficient $1$ in $c_J$).
  This proves the lemma as soon as we show~\eqref{e:RHS_independent}.

 By bilinearity of the intersection form,
 \[\sum_{\{\sigma,\tau\} \in P_{J,v}} \Lambda(\partial(\sigma *
   v),\partial(\tau * v)) = 
   \sum_{\{\sigma,\tau\} \in P_{J,v}} \Lambda'(\partial(\sigma *
      v),\partial(\tau * v))
   =
   \sum_{\{\sigma',\tau'\} \in Q_J}
   a_{\sigma', \tau'} \Lambda'(\sigma',
   \tau'),
 \]
 where $Q_J$ is the set of all (unordered) pairs of distinct $k$-simplices in $J$ and
 $a_{\sigma', \tau'}$ is the number of \emph{appearances} of $\sigma' \subseteq
 \sigma * v$, $\tau' \subseteq \tau * v$
, or $\sigma' \subseteq
 \tau *v$, $\tau' \subseteq \sigma * v$ 
 over all unordered pairs $\{\sigma,
  \tau\} \in P_{J,v}$, modulo $2$. Therefore, for
  checking~\eqref{e:RHS_independent}, it remains to show that $a_{\sigma', \tau'}
  = 1$ if $\{\sigma', \tau'\} \in P_J$ (that is, $\sigma'$ and $\tau'$ are
  disjoint) and $a_{\sigma', \tau'}
  = 0$ if $\{\sigma', \tau'\} \in Q_J \setminus P_J$ ($\sigma'$ and $\tau'$ are
  not disjoint). We also remark that for any $\{\sigma, \tau\} \in P_{J,v}$
  only one of the two options above for appearance is possible, thus we can
  safely assume $\sigma' \subseteq \sigma * v$ and $\tau' \subseteq \tau *
  v$ when counting.

If $\sigma'$ and $\tau'$ share a vertex different from $v$, then there is no
appearance as $\sigma$ and $\tau$ are required to be disjoint and consequently
$\sigma*v$ and $\tau * v$ share only $v$.

If $\sigma'$ and $\tau'$ share $v$ but no other vertex, then there are exactly
two vertices $w_1, w_2$ of $J$ outside $\sigma' \cup \tau'$. Consequently, there are two
appearances $\sigma = (\sigma' - v)* w_1, \tau = (\tau' - v)* w_2$ and
$\sigma = (\sigma' - v)* w_2, \tau = (\tau' - v)* w_1$.

If neither $\sigma'$ nor $\tau'$ contains $v$, then there is the exactly one
appearance: $\sigma = \sigma', \tau = \tau'$. 

  If exactly one of the simplices $\sigma'$, $\tau'$ contains $v$, say $\sigma'$ contains
$v$, then there is exactly one appearance $\sigma = (\sigma' - v)* w$, $\tau =
\tau'$ where $w$ is the vertex of $J$ not in $\sigma' \cup \tau$.
\end{proof}

\begin{proof}[Proof of Proposition~\ref{p:vanKampenForm}(ii)]
We will take $\psi$ and $\Lambda$ in the same way as in the proof of (i).
  It is well known that there is a general position map $g \colon |K| \to
  \R^{2k}$
  such that $\vartheta_g(c_J) = 1$
  with $c_J$
  defined above Lemma~\ref{l:beta_omega_z_J_new}; see, e.g.,~\cite[Example~3.5]{melikhov09}.
  Then by Lemma~\ref{l:beta_omega_z_J_new} and by
  \eqref{e:omega_theta_c} 
  we get
  \[    \sum_{\{\sigma, \tau\} \in P_{J,v}} \Lambda(\partial(\sigma *
	v),\partial(\tau * v)) = \omega_\psi(c_J) = \vartheta_g(c_J) = 1
      \]
      as required.
\end{proof}

\subsection{Proof of Proposition~\ref{p:induction}}

\begin{proof}
  We will prove both items of Proposition~\ref{p:induction}
  simultaneously by induction in the rank of $\Lambda$.
  If $\operatorname{rank}\Lambda = 0$, 
 then $\Lambda(z,z')=0$ for any two cycles $z,z'\in Z_k(K;\Z_2)$. In particular,
  (C\ref{it:kfive}) can only be satisfied in this case if the
  there is no induced subcomplex $J$ of $K$ on $2k+3$
  vertices, i.e., if $n\leq 2k+1$. (Recall that $\Delta_n$ has $n+1$
  vertices.) This yields the base for the induction in both cases. 
 
  Assume now that we are in the case \eqref{e:general_Lambda} and not in \eqref{e:symplectic_Lambda}, which has a better bound.
 Since $\partial \kappa$, where $\kappa$ runs through all $(k+1)$-simplices in $K$, generate $Z_k(K;\Z_2)$, there is a $(k+1)$-simplex $\kappa$ 
  for which $\Lambda(\partial \kappa,\partial \kappa)=1$. Up to
  reordering the vertices, we may assume that $\kappa$ is the simplex on the last $(k+2)$ vertices of $K$.
  Define\footnote{The definition of $\Lambda'$ was obtained as follows. We considered the projection $\pi\colon x\mapsto x-\Lambda(x,\partial\kappa)\partial\kappa$ to the ``orthogonal'' complement of $\partial\kappa$,
  and we took $\Lambda'$ as the pullback of $\Lambda$ under $\pi$.}
 
  \begin{equation}\Lambda'(x,y):=\Lambda(x,y)-\Lambda(x,\partial\kappa)\Lambda(\partial\kappa,y).\label{eq:projectionOne}\end{equation}
    Then $\Lambda'$ is a symmetric bilinear form. If for some $y_0\in Z_k(K;\Z_2)$,
    $\Lambda(x,y_0)=0$ for all $x\in Z_k(K;\Z_2)$, the same is true for
    $\Lambda'$. Moreover $\Lambda'(x,\partial\kappa)=0$ for every $x$, yet $\Lambda(\partial\kappa,\partial\kappa)\neq
    0$. In other words,
    $\operatorname{Ker}\Lambda\subsetneq\operatorname{Ker}\Lambda'$, and so
    $\operatorname{rank}\Lambda'\leq \operatorname{rank}\Lambda-1$.
   Let $K'$ be the subcomplex of $K$ formed by the first $n-k-1$ vertices of
   $K$. We are going to show that $\Lambda'$ restricted to $Z_k(K';\Z_2)$ satisfies both
   (C\ref{it:disjoint}) and (C\ref{it:kfive}) (see
   Proposition~\ref{p:vanKampenForm}).
  Due to (C\ref{it:disjoint}) for $\Lambda$ and \eqref{eq:projectionOne},
  $\Lambda(x,y)=\Lambda'(x,y)$ as long as at least one of the chains $x,y$ is
  disjoint with $\kappa$. This is clearly true when verifying
  (C\ref{it:disjoint}) for $\Lambda'$ on $Z_k(K'; \Z_2)$. It also holds
  when verifying (C\ref{it:kfive}), if $v\neq v_{n-k-1}$.
  However, Lemma~\ref{l:beta_omega_z_J_new}
  then tells us that (C\ref{it:kfive}) holds for $v=v_{n-k-1}$ as well.

 By induction, $n-k-1\leq (2k+1)+(k+1)\operatorname{rank}\Lambda'$, leading to
 \[n\leq (2k+1)+(k+1)\left(\operatorname{rank}\Lambda'+1\right)\leq (2k+1)+(k+1)\operatorname{rank}\Lambda.\]
 
 Let us now prove the case~\eqref{e:symplectic_Lambda}.
 If $\Lambda(z,z)=0$ for every $z\in Z_k(K;\Z_2)$
 and the rank of $\Lambda$ is non-zero, there are two simplices $\kappa$ and $\kappa'$ for which
 $\Lambda(\partial\kappa,\partial\kappa')=1$.
 By reordering the vertices, if necessary, we may assume that $\kappa'$ is the subcomplex on the last $(k+2)$ vertices of $K$.
  Define\footnote{Here we consider the pullback of $\Lambda$ by the
  ``orthogonal'' projection to $\{\partial\kappa,\partial\kappa'\}^\perp$.}
  \begin{equation}\label{eq:projectionTwo}
  \Lambda'(x,y):=\Lambda(x,y) - \Lambda(\partial\kappa,y)\Lambda(x,\partial\kappa')-\Lambda(\partial\kappa',y)\Lambda(x,\partial\kappa).
 \end{equation}
  This is a symmetric bilinear form. If for some $y_0\in Z_k(K;\Z_2)$, $\Lambda(x,y_0)=0$ for all $x\in Z_k(K;\Z_2)$,  the same is true for $\Lambda'$. Moreover $\partial\kappa,\partial\kappa'\in \operatorname{Ker}\Lambda'\setminus\operatorname{Ker}\Lambda$.
  If $\partial\kappa'$ could be written as a linear combination $a\cdot\partial\kappa + b\cdot z$, where $a,b\in\Z_2$ and $z\in\operatorname{Ker}\Lambda$, then
  $1=\Lambda(\partial\kappa,a\partial\kappa+bz)=
  a\Lambda(\partial\kappa,\partial\kappa)+ b\Lambda(\partial\kappa,z)=a\cdot 0 + b\cdot 0 = 0$, a contradiction. 
  It follows that the vectors $\partial\kappa$ and $\partial\kappa'$ are
  linearly independent modulo $\operatorname{Ker}\Lambda$.
  Hence $\dim\operatorname{Ker}\Lambda'\geq \dim\operatorname{Ker}\Lambda+2$ and $\operatorname{rank}\Lambda' \leq \operatorname{rank}\Lambda - 2$.
  
 Let now $K'$ be the complex formed from $K$ by deleting the vertices of $\kappa'$.
 We are going to show that $\Lambda'$ restricted to $Z_k(K';\Z_2)$ satisfies both
  (C\ref{it:disjoint}) and (C\ref{it:kfive}).
 However, if $z,z'\in Z_k(K';\Z_2)$, then both cycles $z$ and $z'$ are disjoint
  with $\kappa'$. Then (C\ref{it:disjoint}) for $\Lambda$ and \eqref{eq:projectionTwo}
 imply that $\Lambda'(z,z')=\Lambda(z,z')$. In particular, $\Lambda'$ satisfies
  both (C\ref{it:disjoint}) and (C\ref{it:kfive}) on $K'$.
 By induction, $n-k-2\leq (2k+1)+\frac{k+2}{2}\operatorname{rank}\Lambda'$, leading to
 \[n\leq (2k+1)+\frac{k+2}{2}\left(\operatorname{rank}\Lambda'+2\right)\leq (2k+1)+\frac{k+2}{2}\operatorname{rank}\Lambda.\qedhere\]
\end{proof}

\section{Completeness}
\label{s:complete}
The aim of this section is to prove Theorem~\ref{t:complete} and then 
Theorem~\ref{t:Z2_computable}. 
Therefore, for
this section, in addition to our standard conventions from
Section~\ref{s:prelim}, we assume that $k \geq
3$ and $M$ is $(k-1)$-connected; unless explicitly stated otherwise, which
occurs only in Remark~\ref{r:homotopy_computable}.

\begin{proof}[Proof of Theorem~\ref{t:complete}]
All considerations in this proof are over $\Z$. According to the statement, we
also assume that we are given
$\psi \colon C_k(K;\Z) \to H_k(M; \Z)$ such that $[\omega_\psi] - \oo(K)$
is trivial. 

  Let $B \subseteq M \setminus \partial M$ be a (closed) $2k$-ball. Assume that $g \colon |K| \to B$ is a general
  position map with $g(|K^{(k-1)}|) \subseteq \partial B$. Our first step will be to find
  a general position map $f' \colon |K| \to (\overline{M \setminus
  B})\setminus \partial M $, agreeing with
  $g$ on $|K^{(k-1)}|$ such that
  $\omega_{f',g} = \omega_\psi$ where $\omega_{f',g}$ is as in Section~\ref{s:transfer}; see~\eqref{e:omega_fg}. The second step will be to find a homotopy of
$f'$ to a general position map $f''$ such that $\vartheta_{f''} = 0$. The third
step will be to remove the remaining self-intersections via standard tricks.

\emph{Step 1.}    We define $f'$ on each $k$-simplex $\sigma \in K$ separately.  We only need
    that $\psi(\sigma) = [f'(\sigma) - g(\sigma)]$. Then $\omega_{f',g} =
    \omega_\psi$ via~\eqref{e:omega_fg}. 
    
By Hurewicz theorem $H_k(M; \Z) \cong \pi_k(M; \Z)$, let $h \colon \pi_k(M) 
 \to H_k(M; \Z)$ be the Hurewicz isomorphism. We also recall the definition of
 $h$ (see~\cite[Chap.3, \S 1.1]{prasolov07}). Given a map $\gamma \colon (S^k,
 s_0) \to (M, x_0)$ where $s_0 \in S^k$, $x_0 \in M$, we set $h(\gamma) :=
 \gamma_*([S^n])$ where $\gamma_*\colon H_k(S^k) \to H_k(M)$ is the induced map
 on homology and $[S^k]$ is the fundamental class. The map $\gamma$ can be also
 regarded as a map from $B^k$ to $M$, constant on $\partial B^k$.

Consider temporarily $\sigma$ as a
simplex in $\R^d$ containing the origin and let $\sigma_\bullet = \frac12 \cdot \sigma$
be a homothetic smaller copy of $\sigma$. Let $f_\bullet \colon
\sigma_\bullet \to M$ be a map, constant on $\partial \sigma_\bullet$,
representing the class $h^{-1}(\psi(\sigma))$ in $\pi_k(M)$. Now we want to
extend $f_\bullet$ to $\sigma$. We have $\overline{\sigma \setminus \sigma_\bullet}
\cong \partial \sigma \times I$, thus we can describe the extension of
$f_\bullet$ on $\partial \sigma \times I$ identifying $\partial \sigma$ with
$\partial \sigma \times \{0\}$ and $\partial \sigma_\bullet$ with $\partial
\sigma \times \{1\}$. Let $f_\bullet$ coincide with $g$ on $\partial
\sigma \times \{0\}$, then we first extend $f_\bullet$ to $\partial
\sigma \times [0,\frac12]$ as a homotopy in $B$ from $g$ to a constant map. Now let
$p \colon [\frac12, 1] \to M$ be an arbitrary path from $f_\bullet(\partial \sigma
\times \{\frac 12\})$ to $f_\bullet(\partial \sigma \times \{1\})$ (recall that
$f_\bullet$ is constant on both $\partial \sigma
\times \{\frac12\}$ and $\partial \sigma \times \{1\}$). For $s \in [\frac12,1]$ we
define $f_\bullet((x,s)) := p(s)$. 

It follows from the construction that the homology class of $f_\bullet(\sigma)
- g(\sigma)$ is $\psi(\sigma)$. Now it is sufficient to consider a homotopy of
$f_\bullet$, constant on $\partial \sigma$, such that the resulting map maps
  the interior of $\sigma$ to $M \setminus (B \cup \partial M)$, and then perform a perturbation to
a required general position map $f'$. 

\emph{Step 2.} From the assumption that $[\omega_{f',g} - \vartheta_g] = [\omega_\psi]
  - \oo (K)$ is trivial and by Lemma~\ref{l:omega}, we get that
  $[\vartheta_{f'}]$ is trivial. By~\eqref{e:sum_finger_moves} this means that 
\[
  \vartheta_{f'} = \sum\limits_{\eta, \mu} n_{\eta,\mu} \delta
  \varepsilon_{\eta,\mu}
\]
  where the sum is over all pairs $(\eta, \mu)$, where $\eta$ is a
  $(k-1)$-simplex, $\mu$ is a $k$-simplex and $\eta \cap \mu = \emptyset$; $\varepsilon_{\eta,\mu}$ are the elementary cochains defined
  above~\eqref{e:sum_finger_moves}; and $n_{\eta,\mu} \in \Z$.
 If $M$ were $\R^{2k}$, then for any $(\mu, \eta)$ we could apply `van Kampen
 finger moves' as described in~\cite[\S 2.4]{freedman-krushkal-teichner94} and
 which provide a homotopy from $f'$ to another map $\hat f$ such that
 $\vartheta_{f'} = \vartheta_{\hat f} \pm \delta\varepsilon_{\eta,\mu}$. (Both choices
$\pm \delta\varepsilon_{\eta,\mu}$ are possible.) In order to adapt to our situation of
 general $M$, we consider a general position PL-path $p$ in the interior
 of $M$ connecting a point in
 the interior of $\eta$ with a point in the interior of $\mu$. Then we consider
 a regular neighborhood $N_p$ of $p$, which is a ball
 by~\cite[Corollary~3.27]{rourke-sanderson72}. We perform the finger-move as
 in~\cite[\S 2.4]{freedman-krushkal-teichner94} inside $N_p$ which has exactly
 same effect on $\vartheta_{f'}$ as in $\R^{2k}$. Therefore we can get a
 homotopy from $f'$ to $f''$ with the required property $\vartheta_{f''} = 0$
 by successively applying finger moves.\footnote{This step of obtaining $f''$ out
 of $f'$ seems to be the bulk of the work~\cite{johnson02}. However, the
 standard approach via finger moves presented here seems to be simpler. (We
 could not directly refer to~\cite{johnson02} in this paragraph, as Johnson
 works in smooth category.)} In addition, $f''(|K|)$ avoids $\partial M$.

 \emph{Step 3.} Finally, we want to build the required embedding $f$ out of $f''$. This
  can be done by standard tricks such as the Whitney trick. They are described
  in~\cite[\S 2.4]{freedman-krushkal-teichner94} for $M = \R^{2k}$. The key
  observation is that all tricks are based on finding a copy of $S^1$ in
  $f(|K|)$ in general position, filling this $S^1$ with a general position disk $D$, 
  taking a regular neighborhood $N_D$ of $D$, which is a ball, and removing the singularities inside $N_D$.
  In a simply connected manifold, these steps work in verbatim. This finishes
  the proof of Theorem~\ref{t:complete}.
\end{proof}

Now, we provide (somewhat weaker) analogy of Theorem~\ref{t:complete} for the
$\Z_2$ case used in the proof of Theorem~\ref{t:Z2_computable}.

\begin{proposition}
  \label{p:independently_even}
  Let us assume that $k \geq 3$ and $M$ is compact $(k-1)$-connected. Then, the
  following conditions are equivalent.
\begin{enumerate}[$(i)$]
\item 
There is a homomorphism $\psi\colon C_k(K; \Z_2) \to H_k(M; \Z_2)$ such that
$[\omega_\psi] - \oo(K) = 0$ (over $\Z_2$).
\item There is a general position map $f'' \colon |K| \to M$ such that for every
  pair $(\sigma, \tau)$ of disjoint $k$-simplices, $f''(\sigma)$ and
  $f''(\tau)$ have an even number of intersections.
\item There is a general position map $f'' \colon |K| \to M$ such that for every
  pair $(\sigma, \tau)$ of $k$-simplices, $f''(\sigma)$ and
  $f''(\tau)$ have an even number of intersections. (We can even assume that
    $f''(\sigma)$ is an embedding on every $k$-simplex $\sigma$ and that
    $f''(\sigma)$ and $f''(\tau)$ share only $f''(\sigma \cap \tau)$, if
    $\sigma$ and $\tau$ are $k$-simplices which are not
  disjoint.)
\end{enumerate}
\end{proposition}

\begin{proof}[Proof of Proposition~\ref{p:independently_even}]
  The implication $(ii) \Rightarrow (i)$ follows
  from Proposition~\ref{p:transfer}. (Any map from a $(k-1)$-complex into a $(k-1)$-connected manifold is nullhomotopic.) The implication $(iii) \Rightarrow (ii)$ is
  obvious.
  
  Thus it remains to prove $(i) \Rightarrow
  (ii)$, and $(ii) \Rightarrow (iii)$. Note that the condition on $f''$ from
  $(ii)$ is equivalent with
  $\vartheta_{f''} = 0$. 
  
  The proof of $(i) \Rightarrow
  (ii)$ is analogous to steps 1. and 2. in the proof of Theorem~\ref{t:complete}, thus we
  only point out the single difference:
  In step 1 for $\Z$ we use the Hurewicz isomorphism $h$; however, we
  only use that $h$ is an epimorphism. If we consider $h$ as a homomorphism $h
  \colon \pi_k(M) \to H_k(M, \Z_2)$ then the proof that $h$ is an epimorphism
  from~\cite[Theorem 3.2]{prasolov07} works in verbatim.

  The proof of $(ii) \Rightarrow (iii)$ follows the step 3 of the proof of
  Theorem~\ref{t:complete}. However, we
  only perform the tricks that remove self-intersections of simplices that share
  at least one vertex.
  (For comparison, the reason why we cannot get rid of all singularities is
  that we cannot perform the Whitney trick. Given two disjoint $k$-simplices
  $\sigma$ and $\tau$ in $K$ the Whitney trick may remove a pair of
  intersection points $\{x, x'\} \subseteq f''(\sigma) \cap f''(\tau)$ provided
  that the signs at $x$ and $x'$ are opposite. But we do not know whether we
  get opposite signs if we perform computations only over $\Z_2$.)
\end{proof}

Now, Theorem~\ref{t:Z2_computable} follows quickly.

\begin{proof}
  By Theorem~\ref{t:quadratic} and Proposition~\ref{p:independently_even}, it is
  sufficient to find out whether the system of equations~\eqref{e:quadratic}
  has a solution in $\Z_2$. This is decidable as $\Z_2$ is finite. 
\end{proof}

\begin{remark}
  \label{r:homotopy_computable}
  Let us consider another algorithmic question: Given a $k$-complex $K$, a compact PL
  $2k$-manifold $M$, not necessarily $(k-1)$-connected, and a general position map $f \colon |K| \to M$. We would
  like to know whether $f$ is homotopic to an embedding. Let us
  also assume that $f$ is presented on the input via its intersection cochain
  $\vartheta_f \in C^{2k}_{\alts}(\tilde K; \Z)$ over the integers.
  Then deducing whether $\oo_f = [\vartheta_f]=0$ is of course
  efficiently computable even over the integers. (In formula~\eqref{e:quadratic}, the term
  $\yy_\sigma^T\AA_{\Omega}\yy_\tau$ on the left side disappears while we have
  $\vartheta_f$ on the right side, thus the equations become linear.)

Therefore, the question whether $f$ is homotopic to an embedding can be solved
efficiently whenever vanishing $\oo_f$ is a complete obstruction for $f$
  being homotopic to an embedding. According to
  Johnson~\cite[Theorem~4]{johnson02} this occurs when $M$ is closed, smooth and
  simply connected. However, the assumption that $M$ is compact, PL and simply connected
  is also sufficient by checking Step~2 of the proof of
  Theorem~\ref{t:complete}; we leave the details for the interested reader.
\end{remark}

\section{Conclusions and open problems}

\label{s:problems}
Here we mention few conclusions and open problems, sometimes touched in the
introduction.

\subparagraph{Existence of the obstruction and completeness.}

Given an almost embedding $f \colon |K| \to M$, the obstruction class $\oo_f$ is well
defined even if we do not assume that $f$ restricted to the
$(k-1)$-skeleton of $K$ is nullhomotopic.  However, we need to assume nullhomotopy for describing the obstruction as in Theorem~\ref{t:obstruction}. In
particular, our approach gives $\Gamma_{K,M} \subseteq \Theta_{K,M}$ where
$\Theta_{K,M} := \{[\omega_\psi] - \oo(K); \psi \in \hom(C_k(K; R), H_k(M;
R))\}$ and $\Gamma_{K,M} := \{\oo_f; f \colon |K| \to M\}$ (considering
only general position PL maps). In particular, if there is an almost embedding 
$f\colon |K| \to M$, then the trivial class belongs to $\Gamma_{K,M}$ and
thereby to $\Theta_{K,M}$ as well, which is in principle our obstruction.

\begin{problem}[Existence]
\label{pr:existence}
Is there an easy to describe superset $\Theta_{K,M}$ of $\Gamma_{K,M}$
  even if we do not assume the nullhomotopy condition, perhaps via (co)homology of $M$ or $K$.
\end{problem}

\begin{problem}[Completeness]
\label{pr:complete}
  When $0 \in \Theta_{K,M}$ implies $0 \in \Gamma_{K,M}$? When $0 \in
  \Gamma_{K,M}$ implies that there is an embedding $f\colon K \to M$?
\end{problem}

If we do not assume that the restriction of every map $f\colon K \to M$
to the $(k-1)$-skeleton is nullhomotopic, the answer to the first question of
Problem~\ref{pr:complete} may of course depend on the answer to
Problem~\ref{pr:existence}. In our proof of Theorem~\ref{t:complete}, the
implication $0 \in \Theta_{K,M} \Rightarrow 0 \in \Gamma_{K,M}$ was the
contents of steps 1 and 2 in the proof and there we really used
$(k-1)$-connectedness of the manifold. The implication $0 \in
  \Gamma_{K,M}$ implies that there is an embedding $f\colon K \to M$ was the
  contents of step 3 and it seems to  be generally well understood. There we
  used $k \geq 3$ and the fact that $M$ is simply-connected. This implication
  does not hold if $k=2$ even if $M =
  \R^{2k}$;~\cite{freedman-krushkal-teichner94}. We also do not expect that the
  requirement that $M$ is simply-connected can be removed in general.

  Somewhat specific case occurs when $k=1$, that is, $K$ is a graph and $M$ is
  a surface, for simplicity connected, otherwise we can treat every
  component separately (let us remark that in this case the nullhomotopy
  condition is satisfied). If $M = \R^2$ then even 
  vanishing the $\Z_2$-version of the van Kampen obstruction implies that $K$
  is a planar graph~\cite{hanani34, tutte70}. When $M$ is a general surface,
  Fulek and Kyn\v{c}l~\cite{fulek-kyncl19} in their noticeable work provide 
  an example of $K$, $M$ and a drawing $f\colon K
  \to M$ such that $\vartheta_f = 0$ over $\Z_2$ whereas $K$ does not embed in
  $M$. This shows that the $\Z_2$-version of our obstruction is not a complete
  obstruction for embeddability of graphs into surfaces. The $\Z$-case is not
  answered yet and it is essentially equivalent to Problem~5.3
  in~\cite{fulek-kyncl19} (restated in our language):
\begin{problem}
  Assume that $K$ is a graph and $M$ a connected orientable surface.
  Assume that there is a general position map $f\colon K \to M$ with
  $\vartheta_f = 0$ (over $\Z$). Does it follow that $K$ embeds in $M$?
\end{problem}

\subparagraph{Computational aspects.} We have already mentioned
Question~\ref{q:e_decidable} in the introduction. Here we only specify a few
concrete cases when $M$ is $(k-1)$-connected and this question seems to be
easiest to approach.

\begin{problem}
  \begin{enumerate}[$(i)$]
  \item
  Is $\emb(k, S^k \times S^k)$ decidable for $k \geq 3$.
  \item
  Is $\emb(4, \HH P^2)$ decidable, where $\HH P^2$ is the quaternionic projective
  plane? (We remark that $\HH P^2$ is an $8$-dimensional manifold.)
  \end{enumerate}
\end{problem}

In the first case the intersection form has matrix 
$
  \AA_\Omega = 
\begin{pmatrix}
0 & 1 \\
(-1)^k & 0 \\
\end{pmatrix}.
$ For $(ii)$, $\AA_\Omega = (1)$.

\subparagraph{Homological almost embeddings.} Motivated by approach
in~\cite{goaoc-patak-patakova-tancer-wagner17} we pose: 

\begin{problem}
Can Theorem~\ref{t:obstruction} be upgraded to homological almost embeddings?
(We refer to~\cite{goaoc-patak-patakova-tancer-wagner17} for a definition
of homological almost embeddings.)
\end{problem}

\section*{Acknowledgments} We would like to thank Xavier Goaoc, Zuzana
Pat\'{a}kov\'{a} and Uli Wagner for discussions in early stages of this
project. We also thank Karim Adiprasito for explaining us the consequences of
his work in~\cite{adiprasito18_arxiv}. 
\bibliographystyle{alpha}
\bibliography{kuhnel2}
\end{document}